\documentclass{amsart}

\usepackage{amsmath}
\usepackage{amsthm}
\usepackage{amssymb}
\usepackage[alphabetic]{amsrefs}
\usepackage[usenames]{xcolor}

\DeclareMathOperator{\mychar}{char}
\DeclareMathOperator{\gal}{Gal}

\DeclareMathOperator{\height}{Ht}

\DeclareMathOperator{\myhom}{Hom}
\DeclareMathOperator{\kernel}{ker}

\DeclareMathOperator{\pr}{pr}
\DeclareMathOperator{\can}{can}

\DeclareMathOperator{\myspan}{span}
\DeclareMathOperator{\rank}{rank}

\DeclareMathOperator{\SL}{SL}
\DeclareMathOperator{\GL}{GL}
\DeclareMathOperator{\PGL}{PGL}
\DeclareMathOperator{\spin}{Spin}
\DeclareMathOperator{\SO}{SO}

\DeclareMathOperator{\ad}{ad}
\DeclareMathOperator{\sss}{ss}
\DeclareMathOperator{\aff}{aff}
\DeclareMathOperator{\tor}{tor}

\DeclareMathOperator{\fixer}{Fix}

\newcommand{ \N } [0] { \mathbf{N} }
\newcommand{ \Z } [0] { \mathbf{Z} }

\newcommand{ \R } [0] { \mathbf{R} }
\newcommand{ \C } [0] { \mathbf{C} }

\newcommand{ \Ga } [0] {\mathbf{G}_{\textup{a}}}
\newcommand{ \Gm } [0] {\mathbf{G}_{\textup{m}}}
\newcommand{ \MU } [0] {\boldsymbol{\mu}}
\newcommand{ \field } [0] {K}
\newcommand{ \integers } [0] {\mathcal{O}}
\newcommand{ \resfield } [0] {k}
\newcommand{ \uniformizer } [0] {\pi}

\newcommand{ \roots } [0] {\Phi}
\newcommand{ \simpleroots } [0] {\Delta}
\newcommand{ \positiveroots } [0] {\Pi}
\newcommand{ \simpleaffineroots } [0] {\simpleroots_{\aff}}
\newcommand{ \simpleaffinegradients } [0] {\widetilde{\simpleroots}}
\newcommand{ \highestroot } [0] {\eta}
\newcommand{ \finiteweylsubscript } [0] {\circ}
\newcommand{ \finiteweylgroup } [0] {W_{\finiteweylsubscript}}
\newcommand{ \prW } [0] {\pr_{\finiteweylsubscript}}
\newcommand{ \canonicalsection } [0] {\mathcal{N}_{\finiteweylsubscript}}
\newcommand{ \goodsectionext } [0] {\mathcal{S}}
\newcommand{ \kottwitzhomcodomain } [0] {\mathbf{\Omega}}
\newcommand{ \lengthset } [0] {\mathcal{L}}
\newcommand{ \flippingset } [0] {\mathcal{F}}
\newcommand{ \dynkinDfive } [5] { \boxed{ \begin{smallmatrix} #1 & #2 & #3 & #4 \\ & & #5 & \end{smallmatrix} } }
\newcommand{ \dynkinEsix } [6] { \boxed{ \begin{smallmatrix} #1 & #2 & #3 & #4 & #5 \\ & & #6 & & \end{smallmatrix} } }
\newcommand{ \orbitsymbol } [0] {\mathcal{O}}

\newcommand{ \building } [0] {\mathcal{B}}

\newcommand{ \alcove } [0] {\mathfrak{A}}
\newcommand{ \barycenter } [0] {\mathfrak{b}}

\newcommand{ \iwahorisubgroup } [0] {\mathcal{I}}
\newcommand{ \prounipotentiwahori } [0] {\mathcal{U}}
\newcommand{ \compactcenter } [0] {Z_c}
\newcommand{ \compacttorus } [0] {T_c}
\newcommand{ \onemodptorus } [0] {T_1}
\newcommand{ \integralmodel } [1] {\mathcal{#1}}
\newcommand{ \specialfiber } [1] {\mathbf{#1}}
\newcommand{ \chargroup } [0] {X^{*}}
\newcommand{ \cochargroup } [0] {X_{*}}

\newcommand{ \del } [0] {\partial}

\newcommand{ \emptyargument } [0] {\hspace{8pt}}

\newcommand{ \defeq } [0] { \stackrel{\textup{\tiny{def}}}{=} }
\newcommand{ \iseq } [0] { \stackrel{\textup{\tiny{?}}}{=} }
\newcommand{ \suchthat } [0] { \hspace{5pt} \vert \hspace{5pt} }
\newcommand{ \canarrow } [0] { \stackrel{ \can }{ \longrightarrow } }
\newcommand{ \directedisom } [0] { \stackrel{ \sim }{ \longrightarrow } }

\newcommand{ \preamble }[1]{\begin{center}\emph{#1}\end{center}}
\newcommand{ \textaftermath } [1] { \hspace{5pt} \text{#1} }

\theoremstyle{plain}
\newtheorem{lemma}{Lemma}[subsection]
\newtheorem{prop}[lemma]{Proposition}
\newtheorem{theorem}[lemma]{Theorem}
\newtheorem{corollary}[lemma]{Corollary}
\newtheorem*{introresult}{Result}

\theoremstyle{definition}
\newtheorem*{defn}{Definition}
\newtheorem*{notation}{Notation}

\theoremstyle{remark}
\newtheorem*{remark}{Remark}
\newtheorem*{example}{Example}
\newtheorem{examplenum}[lemma]{Example}

\begin{document}

\title{On the Canonical Representatives of a Finite Weyl Group}


\author{Sean Rostami}

\address{
\begin{flushleft}
University of Wisconsin\newline
Department of Mathematics\newline
480 Lincoln Dr.\newline
Madison WI 53706
\end{flushleft}
}

\email{srostami@math.wisc.edu}
\email{sean.rostami@gmail.com}

\subjclass[2010]{20G15, 20F05, 20F55, 22E50}


\begin{abstract}
Let $ \field $ be a field and $ G $ a split connected reductive affine algebraic $ \field $-group. Let $ T \subset G $ be a split maximal torus, $ \finiteweylgroup $ its finite Weyl group, and $ \roots $ its root system. After fixing a realization of $ \roots $ in $ G $ and choosing a simple system $ \simpleroots \subset \roots $, one gets a section $ \canonicalsection $ of the canonical map $ N_G(T)(\field) \rightarrow \finiteweylgroup $, whose images are called the \emph{Canonical Representatives}. It is well-known that $ \canonicalsection $ is rarely a group homomorphism, and it is necessary for some questions to understand and quantify this failure. Various new formulas are given which constitute progress in this direction. An application of such formulas to the simple supercuspidals of Gross-Reeder and Reeder-Yu is provided. It seems that more can be said about $ \canonicalsection $, and this will be pursued in subsequent work.
\end{abstract}

\keywords{finite Weyl group, canonical representatives, Tits representatives, Tits group, affine generic character, simple supercuspidal, stable functional, epipelagic}

\maketitle

\tableofcontents

\section*{Introduction}

\subsection*{Motivation and Results}

Fix a field $ \field $. Let $ G $ be a split connected reductive affine algebraic $ \field $-group and $ T \subset G $ a split maximal torus with normalizer $ N_G(T) $ and finite Weyl group $ \finiteweylgroup \defeq N_G(T)(\field) / T(\field) $. Let $ \roots $ be the root system of $ G $ relative to $ T $.

For both computational and theoretical purposes, it is frequently necessary to work with representatives in $ N_G(T)(\field) $ of $ \finiteweylgroup $. A natural example of such a purpose, to determine Bushnell-Kutzko types for certain supercuspidal representations of $ G ( \field ) $ when $ \field $ is $ p $-adic, is given below. For an example from Lie Groups, the algorithmic construction and classification of strong real forms, see \cite{ducloux}.

It is frequently impossible to find a system of representatives for $ \finiteweylgroup $ which forms a \emph{subgroup} of $ N_G(T)(\field) $, i.e. to find a homomorphic section to the canonical map $ N_G(T)(\field) \rightarrow \finiteweylgroup $. This can be seen already for $ G = \SL(2) $ with $ T $ the diagonal torus (if $ \mychar ( \field ) \neq 2 $), but the failure can also occur for adjoint groups.

However, given a sufficiently nice parameterization of $ G $ (in technical terms, a ``based realization of $ \roots $ in $ G $''), which always exists, a well-behaved system of representatives may be defined uniformly. These are called the \emph{Canonical Representatives} and the resulting section of $ N_G(T)(\field) \rightarrow \finiteweylgroup $, usually non-homomorphic, is denoted $ \canonicalsection $. It is desired to understand these representatives as much as possible, especially to understand and quantify the failure of $ \canonicalsection $ to be a group homomorphism.

There is a clean, although not ``closed-form'', description of the obstruction:
\begin{introresult}[\ref{Pgeneralformulafortwococycle}]
Fix a simple system $ \simpleroots \subset \roots $ and let $ \positiveroots \subset \roots $ be the corresponding positive system. If $ u, v \in \finiteweylgroup $ and $ \flippingset ( u, v ) \defeq \{ b \in \positiveroots \suchthat v ( b ) \in - \positiveroots, u ( v ( b ) ) \in \positiveroots \} $ then 
\begin{equation*}
\canonicalsection ( u ) \cdot \canonicalsection ( v ) = \canonicalsection ( u \cdot v ) \cdot \prod_{b \in \flippingset ( u, v )} b^{\vee} ( -1 )
\end{equation*}
\end{introresult}

\begin{remark}
It is perhaps worth mentioning that, in addition to standard universal facts about Coxeter groups, the previous result depends only on two special properties of $ \canonicalsection $, which I isolate here for convenience of the reader: (1) If $ s $ is a simple reflection corresponding to $ \alpha \in \simpleroots $ then $ \canonicalsection ( s )^2 = \alpha^{\vee} ( - 1 ) $, and (2) $ \canonicalsection ( u \cdot v ) = \canonicalsection ( u ) \cdot \canonicalsection ( v ) $ whenever $ \ell ( u \cdot v ) = \ell ( u ) + \ell ( v ) $, where $ \ell $ is the length function for $ \simpleroots $.
\end{remark}

If one is concerned only with the action of $ \canonicalsection ( w ) $ on the root subgroups $ U_a $ ($ a \in \roots $), or if one is content to ignore the center $ Z ( G ) $ for some other reason, then the question becomes more numerical: 

\begin{center}
\emph{The element $ \canonicalsection ( u \cdot v )^{-1} \cdot \canonicalsection ( u ) \cdot \canonicalsection ( v ) \in T ( \field ) $ is, as an automorphism \\of $ U_a $, the scalar $ ( -1 )^{F_{u,v} ( a )} $ for $ F_{u,v} \defeq \sum_{b \in \flippingset ( u, v )} \langle \emptyargument, b^{\vee} \rangle $.}
\end{center}

The ``diagonal obstruction'', $ \flippingset ( w, w ) $ is especially important, as it describes the difference between $ \canonicalsection ( w )^2 $ and $ \canonicalsection ( w^2 ) $. The following general formula is not difficult to prove:
\begin{introresult}[\ref{Pfirstformula}]
Let $ \lengthset ( w ) \defeq \{ b \in \positiveroots \suchthat w ( b ) \in - \positiveroots \} $, and recall that $ \# \lengthset ( w ) $ is the length of $ w $ relative to $ \simpleroots $. 

If $ w \in \finiteweylgroup $ and $ a \in \roots $ then
\begin{equation*}
\sum_{b \in \lengthset ( w )} \langle a, b^{\vee} \rangle = \height ( a ) - \height ( w ( a ) )
\end{equation*}
where $ \height : \roots \rightarrow \Z $ is the height function $ \sum_i m_i \alpha_i \mapsto \sum_i m_i $ ($ \alpha_i \in \simpleroots $).

Read alternatively, the formula expresses $ \sum_{b \in \lengthset ( w )} b^{\vee} $ as a linear combination of $ \hat{\rho} $ and $ w^{-1} ( \hat{\rho} ) $ for $ \hat{\rho} \defeq \frac{1}{2} \sum_{b \in \positiveroots} b^{\vee} $, since $ \height ( \emptyargument ) = \langle \emptyargument, \hat{\rho} \rangle $.
\end{introresult}

If $ | w | \leq 2 $ then this formula is a special case of the diagonal obstruction since, in this case, $ \flippingset ( w, w ) = \lengthset ( w ) $.

If $ | w | > 2 $ then a similar but significantly more complicated formula can be given in an important special case, and likely many other analogous cases:
\begin{introresult}[\ref{Tmainformula}]
Let $ W $ be the extended affine Weyl group of $ G $ relative to $ T $, which has $ \finiteweylgroup $ as a quotient. Let $ \prW : W \rightarrow \finiteweylgroup $ be the obvious projection and let $ \Omega \subset W $ be the stabilizer of the fundamental alcove corresponding to $ \simpleroots $. 

If $ w \in \prW ( \Omega ) $ and $ | w | \geq 3 $ then, for all $ a \in \roots $,
\begin{equation*}
h \cdot \sum_{b \in \flippingset ( w, w )} \langle a, b^{\vee} \rangle = n \cdot [ \height ( a ) - 2 \cdot \height ( w ( a ) ) + \height ( w^2 ( a ) ) ]
\end{equation*}
where $ h $ is the Coxeter Number of $ \roots $ and $ n $ is the size of the fibers of a certain projection map (the details of this projection map are explained in \S\ref{SSformulasforalcovestabilizers}).

Read alternatively, the formula expresses $ \sum_{b \in \flippingset ( w, w )} b^{\vee} $ as a linear combination of $ \hat{\rho} $, $ w^{-1} ( \hat{\rho} ) $, $ w^{-2} ( \hat{\rho} ) $.
\end{introresult}

One application of such formulas is to the topic of ``simple supercuspidals'' of $ G $, which is described next. Assume now that $ \field $ is a $ p $-adic field and denote by $ \C $ the field of complex numbers.

For (almost-)simple and simply-connected split $ G $, \cite{GR} introduces a special kind of supercuspidal representation of $ G ( \field ) $, called a \emph{simple supercuspidal}. Let $ \iwahorisubgroup \subset G ( \field ) $ be the Iwahori subgroup fixing the alcove corresponding to $ \simpleroots $ and let $ \prounipotentiwahori \subset \iwahorisubgroup $ be its pro-unipotent radical. Simple supercuspidals are constructed very concretely by building a character $ \chi : \prounipotentiwahori \cdot Z ( G ) \rightarrow \C^{\times} $, called an \emph{affine generic character}, and compactly inducing $ \chi $ to $ G ( \field ) $. Such representations have arithmetic importance that is explained in more detail by \cite{GR}.

The idea of \cite{GR} is extended by \cite{RY} via the notion of a \emph{stable functional}, which is a suitably generic linear functional $ \lambda $ on a certain finite-dimensional $ \resfield $-vector space $ \mathbf{V} $ coming from the Moy-Prasad filtration at a point $ x $ in the Bruhat-Tits building of $ G $. There are multiple senses in which \cite{RY} is a generalization of \cite{GR}: the group is allowed to be semisimple, it is only required that $ G $ splits over a tamely-ramified extension of $ \field $, and the inducing subgroup is no longer required\footnote{In retrospect, the situation of \cite{RY} specializes mostly to the situation of \cite{GR} when $ x $ is chosen to be the barycenter of an alcove.} to come from an Iwahori subgroup. If $ \lambda $ is such a functional then \cite{RY} composes $ \lambda $ with an additive character $ \resfield \rightarrow \C^{\times} $ and compactly induces the resulting $ \mathbf{V} \rightarrow \C^{\times} $ to get a representation $ \pi $. The representation $ \pi $ splits into a direct sum of irreducible representations which are supercuspidal and, by construction, satisfy the authors' definition of \emph{epipelagic}. Each of these irreducibles is known to be compactly induced from the subgroup $ \fixer ( \lambda ) \subset G ( \field ) $ that fixes $ \lambda $, and so it is worthwhile to understand this fixing subgroup.

Now, assume that $ \lambda $ is such that the resulting $ \mathbf{V} \rightarrow \C^{\times} $ is an affine generic character (this is a situation more general than \cite{GR} but less general than \cite{RY}). It is clear that $ Z ( G ), \prounipotentiwahori \subset \fixer ( \lambda ) $. However, there are $ g \in G ( \field ) $ which stabilize $ \alcove $ (and which therefore fix $ x $ and operate on the domain $ \mathbf{V} $ of $ \lambda $) but for which $ g \notin \iwahorisubgroup $. It turns out that there are very special representatives in $ G ( \field ) $ of each of these elements which indeed fix $ \lambda $:

\begin{introresult}[\ref{Tupdatedtheorem}, \ref{Pproductdecompofcharacterfixer}, \ref{Cfactorizationoflargestcompact}]
There is a section $ \goodsectionext $ of the canonical map $ N_G(T)(\field) \rightarrow W $ such that $ \goodsectionext ( \Omega ) \subset \fixer ( \lambda ) $. Further, $ \fixer ( \lambda ) $ factors as $ \fixer ( \lambda ) = \goodsectionext ( \Omega_{\lambda} ) \cdot \compactcenter \cdot \prounipotentiwahori $ where $ \compactcenter $ is the maximal compact subgroup of $ Z ( G ) ( \field ) $. Finally, $ \fixer ( \lambda ) $ is open and compact-mod-center and its unique maximal compact subgroup $ J $ has the factorization $ J = \goodsectionext ( \Omega_{\tor} ) \cdot \compactcenter \cdot \prounipotentiwahori $, for $ \Omega_{\tor} \subset \Omega $ the torsion subgroup.
\end{introresult}

There various natural ways in which all the results above might be generalized, and I hope to pursue this in a subsequent work.

\subsection*{Outline}

In \S\textbf{\ref{Snotation}}, I merely set some notation/hypotheses and recall some standard notions. In \S\textbf{\ref{Scanonicalrepresentatives}}, I supply the usual definitions of the Canonical Representatives and their most important properties for the convenience of the reader. The longest of the sections, \S\textbf{\ref{Sformulas}}, contains all of the general formulas regarding the failure of the Canonical Section to be homomorphic. The general theme here is that this failure can be expressed by certain linear functionals $ V \rightarrow \R $ on the root space and, surprisingly, these functionals turn out to be very special linear combinations of $ \finiteweylgroup $-orbits of the height function $ \height : V \rightarrow \R $. The precise form of these linear combinations is dictated by the fine additive structure of the root system $ \roots $. In \S\textbf{\ref{Scharacters}}, I introduce some easy cohomological objects and lemmas that will be convenient for the application to simple supercuspidals. Finally, in \S\textbf{\ref{Saffgencharapp}}, I use the material from previous sections to give the above description of $ \fixer ( \lambda ) $. This can be divided roughly into three parts. The first part is to show that the good section $ \goodsectionext $ exists in the case that $ G $ is essentially adjoint--nearly all of the difficulty here is related to the Canonical Representatives. The second part is to ``lift'' the $ \goodsectionext $ to more general $ G $ from the adjoint quotient $ G_{\ad} $--nearly all of the difficulty here is related to the connecting map $ \del $ in the long exact sequence of Galois cohomology coming from $ 1 \rightarrow Z ( G ) \rightarrow G \rightarrow G_{\ad} \rightarrow 1 $. The third part is a straightforward argument showing that $ \goodsectionext ( \Omega ) $ together with some obvious contributions exhausts $ \fixer ( \lambda ) $.

\subsection*{Acknowledgements}

I thank Moshe Adrian for introducing me to the topic of epipelagic representations, which is what eventually led me to study the Canonical Representatives, and for many good conversations about the subject. I thank Nathan Clement for explaining the Hilbert-Mumford Criterion to me, which allowed me to perform the calculations needed in my ongoing attempt to generalize the above results. I thank Daniel Sage and Christopher Bremer for their suggestions following a talk at Louisiana State University. I thank Jeff Adams and Paul Terwilliger for their interest and comments on earlier versions of some of the results below. Finally, the \texttt{MAGMA} software package was used to investigate many things during the resolutions of the above questions.

\section{General Notation and Hypotheses} \label{Snotation}

Let $ \N $ be the monoid of natural numbers. Let $ \Z $ be the ring of integers. Let $ \R $ be the field of real numbers. If $ S $ is a set then $ \# S $ denotes the cardinality of $ S $. If $ f : S \rightarrow S $ is a function and $ T \subset S $ then $ f $ ``stabilizes'' (resp. ``fixes'') $ T $ iff $ f ( t ) \in T $ (resp. $ f ( t ) = t $) for all $ t \in T $. If $ \Gamma $ is a group then its identity element is denoted $ 1_{\Gamma} $. The subgroup of $ g \in \Gamma $ for which there exists $ n \in \N $, $ n \neq 0 $ such that $ g^n = 1_{\Gamma} $ is denoted $ \Gamma_{\tor} $, the torsion subgroup. If $ g \in \Gamma_{\tor} $ then the minimal such $ n $ is denoted $ | g | $, the order of $ g $. Denote by $ \Ga $ and $ \Gm $ the usual $ \Z $-group functors that assign to each commutative ring $ R $ the groups $ R $ and $ R^{\times} $. Similarly, if $ n \in \N $, $ n \neq 0 $ then $ \MU_n $ is the functor that assigns to each commutative ring $ R $ the group $ \{ r \in R^{\times} \suchthat r^n = 1_R \} $.

Let $ \field $ be a non-archimedean local field, and assume that $ \mychar ( \field ) = 0 $. Let $ G $ be a split connected reductive affine algebraic $ \field $-group. Let $ T \subset G $ be a maximal torus that is $ \field $-split. Let $ Z ( G ) $ be the center of $ G $, set $ G_{\ad} \defeq G / Z ( G ) $ (the adjoint quotient), and let $ T_{\ad} \subset G_{\ad} $ be the image of $ T $, a $ \field $-split maximal torus in $ G_{\ad} $. Let $ \chargroup ( T ) $ and $ \cochargroup ( T ) $ be, respectively, the character group and cocharacter group of $ T $. Let $ \langle \emptyargument, \emptyargument \rangle : \chargroup ( T ) \times \cochargroup ( T ) \rightarrow \Z $ be the natural pairing. Sometimes $ \chargroup ( T ) $ and $ \cochargroup ( T ) $ will be abbreviated to $ X $ and $ X^{\vee} $. Similarly, $ \chargroup ( T_{\ad} ) $ and $ \cochargroup ( T_{\ad} ) $ will sometimes be abbreviated to $ X_{\ad} $ and $ X_{\ad}^{\vee} $. Let $ N_G(T) $ be the normalizer of $ T $ and set $ \finiteweylgroup \defeq N_G(T) ( \field ) / T ( \field ) $, the finite Weyl group of $ G $ relative to $ T $.

Let $ \roots \subset X $ be the root system of $ G $ relative to $ T $. \emph{Assume that $ \roots $ is irreducible, i.e. that $ G_{\ad} $ is (almost-)simple.} Let $ V $ be the $ \R $-vector space spanned by $ \roots $ in $ X \otimes_{\Z} \R $ and $ Q \subset V $ the root lattice. Similarly, let $ V^{*} $ be the dual space of $ V $, let $ \roots^{\vee} \subset V^{*} $ be the coroot system, and let $ Q^{\vee} \subset V^{*} $ be the coroot lattice. If $ a \in \roots $ then denote by $ a^{\vee} \in \roots^{\vee} $ the corresponding coroot. The action of $ \finiteweylgroup $ on $ V^{*} $ is the customary dual action $ w ( v^{*} ) \defeq v^{*} \circ w^{-1} $, the pairing $ \langle \emptyargument, \emptyargument \rangle $ is clearly $ \finiteweylgroup $-invariant, and $ w ( a )^{\vee} = w ( a^{\vee} ) $ for all $ a \in \roots $.

In \S\ref{Sformulas} it will be necessary, for $ a, b \in \roots $, to consider the ``$ a $-string through $ b $'' (Proposition 9 Ch VI \S1 no. 3 \cite{bourbakiII}):
\begin{center}
\emph{The set of all $ i \in \Z $ for which $ b + i a \in \roots $ is precisely $ [ -q, p ] \cap \Z $ \\for some $ p, q \in \N $ which necessarily satisfy $ \langle a, b^{\vee} \rangle = q - p $.}
\end{center}

Let $ h $ be the Coxeter Number of $ \finiteweylgroup $. Fix a simple system $ \simpleroots \subset \roots $ and let $ \positiveroots \subset \roots $ be the corresponding positive system. I abuse notation and use the symbol $ \simpleroots $ also to denote the generating set of $ \finiteweylgroup $. The statement ``$ a \in \positiveroots $'' (resp. ``$ a \in - \positiveroots $'') is sometimes abbreviated to ``$ a > 0 $'' (resp. ``$ a < 0 $''). Let $ \ell : \finiteweylgroup \rightarrow \N $ be the length function relative to $ \simpleroots $. For $ w \in \finiteweylgroup $, define $ \lengthset ( w ) \defeq \{ a > 0 \suchthat w ( a ) < 0 \} $, i.e. $ \lengthset ( w ) = \positiveroots \cap w^{-1} ( \positiveroots ) $. Recall that $ \ell ( w ) = \# \lengthset ( w ) $. Let $ \height : \roots \rightarrow \Z $ be the height function relative to $ \simpleroots $, i.e. if $ b = m_1 \alpha_1 + \cdots + m_n \alpha_n $ for $ \alpha_1, \ldots, \alpha_n \in \simpleroots $ then $ \height ( b ) = m_1 + \cdots + m_n $. Denote by $ \highestroot $ the Highest Root, relative to $ \simpleroots $. Finally, set $ \hat{\rho} \defeq \frac{1}{2} \sum_{a \in \positiveroots} a^{\vee} $, and recall that $ \height ( b ) = \langle b, \hat{\rho} \rangle $ (Corollary Ch VI \S1 no. 10 \cite{bourbakiII}).

Let $ \compacttorus \subset T ( \field ) $ be the unique maximal compact open subgroup and define $ W \defeq N_G(T)(\field) / \compacttorus $, the extended affine Weyl group of $ G $ relative to $ T $. Inside the Bruhat-Tits building $ \building_{\ad} $ of $ G_{\ad} $, fix an origin in the apartment for $ T_{\ad} $ and let $ \alcove $ be the alcove at the origin corresponding to the simple system $ \simpleroots $. Via the quotient map $ G ( \field ) \rightarrow G_{\ad} ( \field ) $, $ G ( \field ) $ acts on $ \building_{\ad} $ and $ W $ acts on the apartment for $ T_{\ad} $. Let $ W_{\aff} \subset W $ be the affine Weyl group and let $ \Omega \subset W $ be the subgroup that stabilizes $ \alcove $. It is standard that $ W \cong W_{\aff} \rtimes \Omega $ and, on the other hand, that $ W \cong X^{\vee} \rtimes \finiteweylgroup $. Let $ \prW : W \rightarrow \finiteweylgroup $ be the projection, which will frequently be applied to $ \Omega $. The alcove $ \alcove $ gives $ W_{\aff} $ the structure of a Coxeter group: the union of the reflections in $ \simpleroots $ together with the reflection across the nullspace of $ 1 - \highestroot $ is a Coxeter generating set $ \simpleaffineroots $. I abuse notation and also use $ \simpleaffineroots $ to denote $ \simpleroots \cup \{ 1 - \highestroot \} $. It will also be convenient to set $ \simpleaffinegradients \defeq \simpleroots \cup \{ - \highestroot \} $, the ``gradients'' of $ \simpleaffineroots $. Denote by $ \barycenter \in \alcove $ the barycenter, which is the unique point at which the value $ a ( \barycenter ) $ is independent of $ a \in \simpleaffineroots $. It is clear from the definitions that $ \Omega $ permutes $ \simpleaffineroots $ and therefore fixes $ \barycenter $.

Denote by $ \Omega_{\ad} $ the object for $ G_{\ad} $ that is analogous to $ \Omega $, i.e. the stabilizer of $ \alcove $ in the extended affine Weyl group of $ G_{\ad} $ relative to $ T_{\ad} $. It can be proved that if $ \sigma \in \finiteweylgroup $ permutes $ \simpleaffinegradients $ then $ \sigma = \prW ( \omega ) $ for some $ \omega \in \Omega_{\ad} $. There is a canonical group homomorphism $ \Omega \rightarrow \Omega_{\ad} $ which is compatible with the actions of $ G ( \field ) $ and $ G_{\ad} ( \field ) $ on the Bruhat-Tits building of $ G_{\ad} $, and the image of $ \omega \in \Omega $ in $ \Omega_{\ad} $ will be denoted by $ \omega_{\ad} $.

Finally, consider the following two $ \Omega_{\ad} $-actions:
\begin{itemize}
\item Via the isomorphism $ \Omega_{\ad} \directedisom X_{\ad}^{\vee} / Q^{\vee} $ and item (XII) of the Plates at the end of \cite{bourbakiII}, $ \Omega_{\ad} $ transforms the Affine Dynkin Diagram of $ \roots $.

\item Via the projection $ \prW : \Omega_{\ad} \rightarrow \finiteweylgroup $, $ \Omega_{\ad} $ permutes $ \simpleaffinegradients $.
\end{itemize}
It will be important to remember that, by Ch VI \S4 no. 3 \cite{bourbakiII}, the vertex set of the Affine Dynkin Diagram is precisely $ \simpleaffinegradients $ and the induced permutation of the vertices by $ \omega \in \Omega_{\ad} $ is precisely the permutation of $ \simpleaffinegradients $ by $ \prW ( \omega ) \in \finiteweylgroup $.

\section{Realizations and Canonical Representatives} \label{Scanonicalrepresentatives}

For each $ a \in \roots $, denote by $ U_a \subset G $ the unique subgroup that is the image of an injective morphism $ \Ga \rightarrow G $ and on which $ T $ acts, in coordinates, via $ a $.

Recall\footnote{\S8.1 \cite{springer}} that a ``realization of $ \roots $ in $ G $'' is a collection of isomorphisms $ u_a : \Ga \directedisom U_a $ satisfying certain properties, including the property that for $ a \in \simpleroots $ the element $ n_a \defeq u_a ( 1 ) \cdot u_{-a} ( -1 ) \cdot u_a ( 1 ) \in N_G(T)(\field) $ represents the generator $ s_a \in \finiteweylgroup $.

\begin{notation}
For the rest of the paper, a realization $ \{ u_a \}_{a \in \roots} $ of $ \roots $ in $ G $ is fixed and $ n_a $ is defined as above for all $ a \in \simpleroots $.
\end{notation}

Because the representatives $ n_a $ ($ a \in \simpleroots $) satisfy\footnote{Proposition 9.3.2 \cite{springer}} a ``braid relation'', the element $ n_{a_1} \cdots n_{a_{\ell}} \in N_G(T)(\field) $ is the same for all reduced expressions $ w = s_1 \cdots s_{\ell} $ ($ s_i $ the reflection for $ a_i \in \simpleroots $). Hence, one may extend the representatives to all elements of $ \finiteweylgroup $ by setting $ n_w \defeq n_{a_1} \cdots n_{a_{\ell}} $ for any/all reduced expressions $ w = s_1 \cdots s_{\ell} $.

\begin{defn}[Canonical Section]
The section of the canonical map $ N_G(T)(\field) \rightarrow \finiteweylgroup $ defined by $ w \mapsto n_w $ is called the \emph{Canonical Section} and denoted $ \canonicalsection : \finiteweylgroup \rightarrow N_G(T)(\field) $.
\end{defn}

Two key properties of $ \canonicalsection $ are:
\begin{itemize}
\item $ \canonicalsection ( s )^2 = a^{\vee} ( -1 ) $ for $ s $ the reflection corresponding to $ a \in \simpleroots $.
\item If $ \ell ( u \cdot v ) = \ell ( u ) + \ell ( v ) $ then $ \canonicalsection ( u \cdot v ) = \canonicalsection ( u ) \cdot \canonicalsection ( v ) $.
\end{itemize}

It is well-known that $ \canonicalsection $ is almost never a group homomorphism, but it is quite well-behaved nonetheless. The ultimate goal is to understand as much as possible the failure of $ \canonicalsection $ to be a group homomorphism. In homological terms, the goal is to understand as precisely as possible the $ 2 $-cocycle
\begin{align}
\finiteweylgroup \times \finiteweylgroup &\longrightarrow T(\field) \label{Ecocycle} \\
\nonumber ( u, v ) &\longmapsto \canonicalsection ( u \cdot v )^{-1} \cdot \canonicalsection ( u ) \cdot \canonicalsection ( v )
\end{align}


\begin{defn}
For each $ n \in N_G(T)(\field) $, representing $ w \in \finiteweylgroup $, and $ a \in \roots $, conjugation by $ n $ is an isomorphism $ U_a \rightarrow U_{ w ( a ) } $ which is, in realization-coordinates $ \Ga \rightarrow \Ga $, a scalar. Define $ c ( n, a ) \in \field^{\times} $ to be this scalar.
\end{defn}

Since $ G $ is split, it is always possible\footnote{Proposition 9.5.3 \cite{springer}} to choose a realization of $ \roots $ in $ G $ which is ``as simple as possible'', which includes the following important property:

\begin{center}
\emph{$ c ( n, a ) = \pm 1 $ for all $ a \in \roots $ and if $ a, w ( a ) \in \simpleroots $ then $ c ( n, a ) = 1 $.}
\end{center}

Such a realization is called a ``Chevalley Realization''. However, this will not be assumed until \S\ref{Scharacters}.

\section{Formulas for the Cocycle} \label{Sformulas}

\subsection{General results}

\begin{lemma}[Exchange Property for Canonical Representatives] \label{Lexchangeproperty}
Fix $ w \in \finiteweylgroup $ and $ \alpha \in \simpleroots $. Let $ s \in \finiteweylgroup $ be the reflection for $ \alpha $. \textbf{Assertion:} If $ \ell ( w \cdot s ) < \ell ( w ) $ then 
\begin{equation*}
\canonicalsection ( w ) \cdot \canonicalsection ( s ) = \canonicalsection ( w \cdot s ) \cdot \alpha^{\vee} ( -1 )
\end{equation*}
\end{lemma}

\begin{proof}
Set $ n = \ell ( w ) $ and write $ w = s_1 \cdots s_n $ for some $ s_i \in \simpleroots $. By the Exchange Property for $ ( \finiteweylgroup, \simpleroots ) $, there is $ i $ such that $ w = s_1 \cdots s_{i-1} \cdot s_{i+1} \cdots s_n \cdot s $. Since this is a reduced word for $ w $, and since $ \canonicalsection $ is homomorphic on reduced words, $ \canonicalsection ( w ) = \canonicalsection ( s_1 ) \cdots \canonicalsection ( s_{i-1} ) \cdot \canonicalsection ( s_{i+1} ) \cdots \canonicalsection ( s_n ) \cdot \canonicalsection ( s ) $. Similarly, $ \canonicalsection ( w \cdot s ) = \canonicalsection ( s_1 ) \cdots \canonicalsection ( s_{i-1} ) \cdot \canonicalsection ( s_{i+1} ) \cdots \canonicalsection ( s_n ) $. Thus, $ \canonicalsection ( w ) = \canonicalsection ( w \cdot s ) \cdot \canonicalsection ( s ) $. Since $ \canonicalsection $ has the property $ \canonicalsection ( s )^2 = \alpha^{\vee} ( -1 ) $, the claim follows.
\end{proof}

Recall the following fact\footnote{Exercise \#12(e) Ch IX \S4 \cite{bourbakiIII}} regarding $ \canonicalsection $:
\begin{center}
\emph{If $ w^2 = 1 $ then $ \canonicalsection ( w )^2 = \prod_{a \in \lengthset ( w )} a^{\vee} ( -1 ) $.}
\end{center}

It happens that a full generalization is possible with very little extra work:
\begin{defn}[Flipping Set]
For $ u, v \in \finiteweylgroup $,
\begin{equation*}
\flippingset ( u, v ) \defeq \{ a > 0 \suchthat v ( a ) < 0, u ( v ( a ) ) > 0 \}
\end{equation*}
\end{defn}

\begin{prop}[Cocycle Formula] \label{Pgeneralformulafortwococycle}
For $ u, v \in \finiteweylgroup $,
\begin{equation*}
\canonicalsection ( u ) \cdot \canonicalsection ( v ) = \canonicalsection ( u \cdot v ) \cdot \prod_{a \in \flippingset ( u, v )} a^{\vee} ( -1 )
\end{equation*}
\end{prop}

This generalizes the previous fact since $ w^2 = 1 $ implies that $ \flippingset ( w, w ) =  \lengthset ( w ) $.

\begin{proof}
Let $ n = \ell ( v ) $ and write $ v = s_1 \cdots s_n $ for some $ s_i \in \simpleroots $. For each $ i $, let $ \alpha_i $ be the simple root whose reflection is $ s_i $. For each $ i \in \{ 1, 2, \ldots, n \} $ define $ \theta_i = s_n \cdots s_{i+1} ( \alpha_i ) $, and recall\footnote{Corollary 2 Ch VI \S1 no. 6 \cite{bourbakiII}} that $ \lengthset ( w ) = \{ \theta_1, \ldots, \theta_n \} $. \emph{To be clear, $ \theta_n = \alpha_n $.} Define $ I $ to be the set of all $ i \in \{ 1, 2, \ldots, n \} $ such that $ \ell ( u \cdot s_1 \cdots s_{i-1} \cdot s_i ) < \ell ( u \cdot s_1 \cdots s_{i-1} ) $. \emph{To be clear, $ 1 \in I $ if and only if $ \ell ( u \cdot s_1 ) < \ell ( u ) $.} Set $ \Theta = \{ \theta_i \suchthat i \in I \} $. By repeated use of Lemma \ref{Lexchangeproperty} on $ \canonicalsection ( u ) \cdot \canonicalsection ( v ) = \canonicalsection ( u ) \cdot \canonicalsection ( s_1 ) \cdots \canonicalsection ( s_n ) $, I conclude that $ \canonicalsection ( u ) \cdot \canonicalsection ( v ) = \canonicalsection ( u \cdot v ) \cdot \prod_{\theta \in \Theta} \theta^{\vee} ( -1 ) $. Therefore, the claim is $ \Theta = \flippingset ( u, v ) $. 

Let $ a \in \flippingset ( u, v ) $ be arbitrary. By definition of $ \flippingset ( u, v ) $, there is $ i \in \{ 1, 2, \ldots, n \} $ such that $ s_i \cdots s_n ( a ) \in - \positiveroots $ but $ s_{i+1} \cdots s_n ( a ) \in \positiveroots $. \emph{To be clear, it may happen that $ s_n ( a ) \in - \positiveroots $, which is what is meant in the case that $ i = n $.} Since $ s_i $ permutes $ \positiveroots \setminus \{ \alpha_i \} $, it must be true that $ s_{i+1} \cdots s_n ( a ) = \alpha_i $ and so $ a = s_n \cdots s_{i+1} ( \alpha_i ) = \theta_i $. \emph{To be clear, in the case that $ i = n $, the conclusion is that $ a = \alpha_n $.} To show that $ a \in \Theta $, I must show that $ \ell ( u \cdot s_1 \cdots s_{i-1} \cdot s_i ) < \ell ( u \cdot s_1 \cdots s_{i-1} ) $. This is equivalent\footnote{Lemma (b) \S1.6 \cite{humphreysCG}} to the statement that $ u \cdot s_1 \cdots s_{i-1} ( \alpha_i ) \in - \positiveroots $, which is immediate from the definitions:
\begin{equation*}
u \cdot s_1 \cdots s_{i-1} ( \alpha_i ) = - u \cdot s_1 \cdots s_{i-1} \cdot s_i ( \alpha_i ) = - u \cdot s_1 \cdots s_n ( a ) = - u ( v ( a ) ) \in - \positiveroots
\end{equation*}
This establishes that $ \Theta \supset \flippingset ( u, v ) $.

Let $ \theta \in \Theta $ be arbitrary. Since $ \theta \in \positiveroots $ and $ v ( \theta ) \in - \positiveroots $, I must show only that $ u ( v ( \theta ) ) \in \positiveroots $. Let $ i \in I $ be such that $ \theta = \theta_i $. Note that  
\begin{equation*}
u ( v ( \theta ) ) = u ( v ( s_n \cdots s_{i+1} ( \alpha_i ) ) ) = u ( s_1 \cdots s_i ( \alpha_i ) ) = - u ( s_1 \cdots s_{i-1} ( \alpha_i ) )
\end{equation*}
So, to show $ u ( v ( \theta ) ) \in \positiveroots $, it is equivalent to show $ u ( s_1 \cdots s_{i-1} ( \alpha_i ) ) \in - \positiveroots $. This is equivalent\footnote{Lemma (b) \S1.6 \cite{humphreysCG}} to the statement that $ \ell ( u \cdot s_1 \cdots s_{i-1} \cdot s_i ) < \ell ( u \cdot s_1 \cdots s_{i-1} ) $, which is true by definition of $ I $. This establishes that $ \Theta \subset \flippingset ( u, v ) $.
\end{proof}

\begin{example}
Suppose $ u, v \in \simpleroots $ are distinct and let $ \alpha, \beta \in \simpleroots $ be such that $ u ( \alpha ) = - \alpha $ and $ v ( \beta ) = - \beta $. Since $ v $ permutes $ \positiveroots \setminus \{ \beta \} $, $ \flippingset ( u, v ) \subset \{ \beta \} $. Since $ \beta \neq \alpha $ and $ u $ permutes $ \positiveroots \setminus \{ \alpha \} $, $ u ( \beta ) \in \positiveroots $. Thus, $ u ( v ( \beta ) ) = u ( - \beta ) = - u ( \beta ) \in - \positiveroots $ and therefore $ \flippingset ( u, v ) = \emptyset $. This is as expected since $ \canonicalsection ( u \cdot v ) = \canonicalsection ( u ) \cdot \canonicalsection ( v ) $ whenever $ \ell ( u \cdot v ) = \ell ( u ) + \ell ( v ) $.
\end{example}

The ultimate goal is to understand $ \flippingset ( u, v ) $ to the greatest extent possible. However, since the intersection of $ \kernel ( a ) $ for all $ a \in \roots $ is $ Z ( G ) $, to understand the difference between $ \canonicalsection ( u \cdot v ) $ and $ \canonicalsection ( u ) \cdot \canonicalsection ( v ) $ as elements of $ N_G(T) $ is largely the same as to understand, for all $ a \in \roots $, the difference between the two isomorphisms $ U_a \directedisom U_{u(v(a))} $ that they induce. For this reason, certain linear functionals come into play:
\begin{notation}
$ F_{u,v} \defeq \sum_{ b \in \flippingset ( u, v ) } \langle \emptyargument, b^{\vee} \rangle $
\end{notation}

According to Proposition \ref{Pgeneralformulafortwococycle}, the difference between $ \canonicalsection ( u \cdot v ) $ and $ \canonicalsection ( u ) \cdot \canonicalsection ( v ) $ on $ U_a $ is $ ( -1 )^{F_{u,v} ( a )} $. Therefore, the goal is largely to understand the functional $ F_{u,v} $.

\subsection{The flipping set for iteration of an element}

Let $ w \in \finiteweylgroup $ be arbitrary.

\begin{notation}
$ \flippingset ( w ) \defeq \flippingset ( w, w ) = \{ a > 0 \suchthat w ( a ) < 0, w^2 ( a ) > 0 \} $. Analogously, $ F_w \defeq \sum_{b \in \flippingset ( w ) } \langle \emptyargument, b^{\vee} \rangle $.
\end{notation}
Note that this set describes the difference between $ \canonicalsection ( w^2 ) $ and $ \canonicalsection ( w )^2 $, i.e. it describes the restriction of the cocycle (\ref{Ecocycle}) to the diagonal.

Since $ \flippingset ( w ) = \lengthset ( w ) $ whenever $ | w | \leq 2 $, the following can be considered a ``first approximation'' to the goal:
\begin{prop}[Summation Formula \#1] \label{Pfirstformula}
For all $ w \in \finiteweylgroup $ and all $ a \in \roots $,
\begin{equation*}
\sum_{b \in \lengthset ( w )} \langle a, b^{\vee} \rangle = \height ( a ) - \height ( w ( a ) )
\end{equation*}
\end{prop}

Observe that the well-known fact $ \langle a, \hat{\rho} \rangle = \height ( a ) $, for $ \hat{\rho} \defeq \frac{1}{2} \sum_{a > 0} a^{\vee} $, is recovered from this by choosing $ w $ to be the Longest Element $ w_0 $, since $ \lengthset ( w_0 ) = \positiveroots $ and $ \height ( w_0 ( a ) ) = - \height ( a ) $.

\begin{proof}
It follows from repeated use of Proposition 29(ii) Ch VI \S1 no. 10 \cite{bourbakiII} that $ w^{-1} ( \hat{\rho} ) = \hat{\rho} - \sum_{\theta \in \lengthset ( w ) } \theta^{\vee} $. Recalling that $ \langle a, \hat{\rho} \rangle = \height ( a ) $ and using $ \finiteweylgroup $-invariance of the pairing finishes the proof.
\end{proof}

\begin{remark}
The above formula explains some numerical curiosities in the proof (by direct computation) of the Theorem from \cite{roro}; see the proof of Corollary \ref{Csumeven} below.
\end{remark}

Here is an easy symmetry that will be useful later:
\begin{lemma}[Forward-Backward Symmetry] \label{Lsymmetry}
$ w^2 ( \flippingset ( w ) ) = \flippingset ( w^{-1} ) $ and so, for all $ a \in \roots $, $ F_w ( w^{-1} ( a ) ) = F_{w^{-1}} ( w ( a ) ) $.
\end{lemma}

\begin{proof}
From the definition, $ \flippingset ( w ) = \positiveroots \cap w^{-1} ( - \positiveroots ) \cap w^{-2} ( \positiveroots ) $. Since injective functions distribute across intersections, applying $ w^2 $ to this yields $ w^2 ( \flippingset ( w ) ) = w^2 ( \positiveroots ) \cap w ( - \positiveroots ) \cap \positiveroots = \flippingset ( w^{-1} ) $. This rearranges to $ w ( \flippingset ( w ) ) = w^{-1} ( \flippingset ( w^{-1} ) ) $. The second identity follows by summing $ \langle a, \emptyargument \rangle $ over these two latter sets and using $ \finiteweylgroup $-invariance of the pairing.
\end{proof}

Recalling that $ \lengthset ( w ) = \positiveroots \cap w^{-1} ( - \positiveroots ) $, the following presentation of $ \flippingset ( w ) $ highlights the ``recursive'' nature of flipping sets:

\begin{lemma}
$ \flippingset ( w ) = \lengthset ( w ) \cap w^{-1} ( - \lengthset ( w ) ) $.
\end{lemma}

\begin{proof}
$ \lengthset ( w ) \cap w^{-1} ( - \lengthset ( w ) ) $ is $ \{ a > 0 \suchthat w ( a ) < 0, w ( a ) \in - \lengthset ( w ) \} $, which is the same as $ \{ a > 0 \suchthat w ( a ) < 0, - w ( a ) > 0, w ( - w ( a ) ) < 0 \} $, which is the same as $ \{ a > 0 \suchthat w ( a ) < 0, w ( w ( a ) ) > 0 \} = \flippingset ( w ) $.
\end{proof}

\subsection{Explicit formulas for alcove-stabilizers} \label{SSformulasforalcovestabilizers}

In the rest of \S\ref{Sformulas}, namely this \S\ref{SSformulasforalcovestabilizers} and \S\ref{SSproofofconstantfibers} next, it is assumed that $ w \in \prW ( \Omega ) $ and that $ | w | \geq 3 $. Again, this latter assumption is justified because if $ | w | \leq 2 $ then $ \flippingset ( w ) = \lengthset ( w ) $ and so Proposition \ref{Pfirstformula} (Formula \#1) applies.

\begin{lemma} \label{Lordersameorbitmaximal}
If $ \omega \in \Omega $, $ \sigma := \prW ( \omega ) $, and $ \orbitsymbol $ is the $ \omega $-orbit of the simple affine root $ 1 - \highestroot $, then $ \vert \sigma \vert = \vert \omega_{\ad} \vert = \# \orbitsymbol $.
\end{lemma}

\begin{proof}
It is obvious from the semidirect product that $ \omega_{\ad}^n = 1 $ implies $ \sigma^n = 1 $. Conversely, if $ \sigma^n = 1 $ then, by the semidirect product, $ \omega_{\ad}^n \in X_{\ad}^{\vee} $. But $ \omega_{\ad}^n $ stabilizes the alcove $ \alcove $ and no translation can do this except translation by $ 0 $, so $ \sigma^n = 1 $ implies $ \omega_{\ad}^n = 1 $. This establishes the first equality. If $ N := \# \orbitsymbol $ then $ \omega^N ( 1 - \highestroot ) = 1 - \highestroot $ and therefore $ \omega^N $ permutes $ \simpleroots $. This implies that $ \omega^N $ fixes the origin, so $ \omega_{\ad}^N \in \finiteweylgroup $. But only $ 1 \in \finiteweylgroup $ can stabilize $ \simpleroots $ so $ \omega_{\ad}^N = 1 $ and therefore $ N = \vert \omega_{\ad} \vert $.
\end{proof}

\begin{lemma} \label{Lmultiplicitiespreserved}
Enumerate $ \simpleroots = \{ \alpha_1, \ldots, \alpha_{\ell} \} $ and write $ \highestroot = m_1 \alpha_1 + \cdots + m_{\ell} \alpha_{\ell} $. Set $ \alpha_0 := - \highestroot $ and $ m_0 := 1 $ so that $ \sum_i m_i \alpha_i = 0 $. \textbf{Assertion:} If $ \omega \in \Omega $ and $ \sigma := \prW ( \omega ) $ then $ m_i = m_j $ whenever $ \sigma ( \alpha_i ) = \alpha_j $.
\end{lemma}

\begin{proof}
I may assume that $ \sigma \neq 1 $, in which case Lemma \ref{Lordersameorbitmaximal} implies that there is $ 1 \leq j \leq \ell $ such that $ \sigma ( \alpha_j ) = - \highestroot $. Applying $ \sigma $ to $ \sum_i m_i \alpha_i = 0 $, isolating $ m_j \sigma ( \alpha_j ) = - m_j \highestroot $, and dividing by $ m_j $ yields $ \highestroot = \sum_{i \neq j} \frac{m_i}{m_j} \sigma ( \alpha_i ) $. Since $ \simpleroots $ is linearly-independent and $ \sigma ( \alpha_i ) \in \simpleroots $ for all $ i \neq j $, $ m_i / m_j \in \N $ for all $ i $. Since $ m_0 = 1 $, it must then be true that $ m_j = 1 $. Thus, $ m_1 \alpha_1 + \cdots + m_{\ell} \alpha_{\ell} = \highestroot = \sum_{i \neq j} m_i \sigma ( \alpha_i ) $ and comparison of the coefficients finishes the proof.
\end{proof}

\begin{notation}
Let $ R, S \in \simpleroots $ be such that $ w ( S ) = R $ and $ w ( R ) = - \highestroot $.
\end{notation}

Such $ R, S $ exist by Lemma \ref{Lordersameorbitmaximal} since $ | w | \geq 3 $.

\begin{notation}
If $ b \in \roots $ and $ \alpha \in \simpleroots $ then denote by $ b_{\alpha} $ the $ \alpha $-coefficient of $ b $ when expressed using the basis $ \simpleroots $.
\end{notation}

\begin{corollary} \label{CcoeffsRSonly01}
$ b_R, b_S \in \{ -1, 0, 1 \} $ for all $ b \in \roots $.
\end{corollary}

\begin{proof}
By definition of ``highest'', it suffices to prove the claim when $ b = \highestroot $, but this is immediate from Lemma \ref{Lmultiplicitiespreserved}.
\end{proof}

\begin{examplenum} \label{EXfirstD5example}
Let $ G $ be the adjoint group D5, so $ \Omega = \Omega_{\ad} \cong \Z / 4 \Z $. Following Plate IV \cite{bourbakiII}, enumerate $ \simpleroots = \{ \alpha_1, \alpha_2, \alpha_3, \alpha_4, \alpha_5 \} $ and depict any $ m_1 \alpha_1 + m_2 \alpha_2 + m_3 \alpha_3 + m_4 \alpha_4 + m_5 \alpha_5 \in \roots $ in ``Dynkin format'' by $ \dynkinDfive{m_1}{m_2}{m_3}{m_4}{m_5} $. From Plate IV (XII) \cite{bourbakiII}, there is a generator $ w \in \prW ( \Omega ) $ whose permutation of $ \simpleaffinegradients $ has the disjoint cycle decomposition $ ( \alpha_1, \alpha_4, - \highestroot, \alpha_5 ) ( \alpha_2, \alpha_3 ) $. Thus, $ S = \dynkinDfive{1}{0}{0}{0}{0} $ and $ R = \dynkinDfive{0}{0}{0}{1}{0} $.
\end{examplenum}

It is important to observe that the existence of $ w \in \prW ( \Omega ) $ such that $ | w | \geq 3 $ forces $ \roots $ to be \emph{simply-laced}. The various equivalent expressions of what it means to be simply-laced are completely standard, but there does not seem to be a proof in print so I record one here:

\begin{lemma}[Simply-Laced] \label{Lsimplylaced}
Recall that $ \roots $ is irreducible and reduced. The following three properties are equivalent:
\begin{enumerate}
\item The Dynkin Diagram of $ \roots $ has no multiple bonds.\label{Lsimplylaced_nomultbonds}

\item $ \langle a, b^{\vee} \rangle \in \{ -1, 0, 1 \} $ for all $ a, b \in \roots $ such that $ a \notin \R \cdot b $.\label{Lsimplylaced_nobigpairings}

\item The Cartan Matrix of $ \roots $ is symmetric.\label{Lsimplylaced_cartansymmetric}

\item $ || a || = || b || $ for all $ a, b \in \roots $.\label{Lsimplylaced_samelength}
\end{enumerate}

$ \roots $ is called ``simply-laced'' iff it has any/all of these properties.
\end{lemma}

Here $ || \emptyargument || $ denotes the norm defined by some fixed $ \finiteweylgroup $-invariant inner product $ ( \emptyargument | \emptyargument ) $ on $ V $, which necessarily satisfies $ \langle a, b^{\vee} \rangle = 2 ( a | b ) / ( b | b ) $ for all $ a, b \in \roots $.

\begin{proof}
\textbf{(\ref{Lsimplylaced_samelength}) $ \boldsymbol{\Rightarrow} $ (\ref{Lsimplylaced_cartansymmetric}), (\ref{Lsimplylaced_nobigpairings}):} Since $ \langle a, b^{\vee} \rangle = 2 ( a | b ) / ( b | b ) $ and $ \langle b, a^{\vee} \rangle = 2 ( b | a ) / ( a | a ) $, combining symmetry $ ( a | b ) = ( b | a ) $ with (\ref{Lsimplylaced_samelength}) proves that $ \langle a, b^{\vee} \rangle = \langle b, a^{\vee} \rangle $ for all $ a, b \in \roots $, which implies (\ref{Lsimplylaced_cartansymmetric}). If, further, $ a \notin \R \cdot b $ then the table in Ch VI \S1 no. 3 \cite{bourbakiII} shows that $ -1, 0, 1 $ are the only possible values for $ \langle a, b^{\vee} \rangle $. \textbf{(\ref{Lsimplylaced_nobigpairings}) $ \boldsymbol{\Rightarrow} $ (\ref{Lsimplylaced_nomultbonds}):} This is trivial, since the number of bonds between vertices $ a, b \in \simpleroots $ is $ \langle a, b^{\vee} \rangle \cdot \langle b, a^{\vee} \rangle $. \textbf{(\ref{Lsimplylaced_nomultbonds}) $ \boldsymbol{\Rightarrow} $ (\ref{Lsimplylaced_samelength}):} Fix $ \alpha, \beta \in \simpleroots $. Suppose first that $ \langle \alpha, \beta^{\vee} \rangle \neq 0 $, i.e. that $ \alpha, \beta $ are adjacent vertices in the Dynkin Diagram. By (\ref{Lsimplylaced_nomultbonds}), $ \langle \alpha, \beta^{\vee} \rangle \cdot \langle \beta, \alpha^{\vee} \rangle = 1 $, and so $ \langle \alpha, \beta^{\vee} \rangle = \langle \beta, \alpha^{\vee} \rangle = \pm 1 $. Symmetry $ ( a | b ) = ( b | a ) $ and the identities $ \langle a, b^{\vee} \rangle = 2 ( a | b ) / ( b | b ) $ and $ \langle b, a^{\vee} \rangle = 2 ( b | a ) / ( a | a ) $ imply $ || \alpha || = || \beta || $. By traversing the (connected) Dynkin Diagram, $ || \alpha || $ is the same for all $ \alpha \in \simpleroots $. Call this common value $ L $. Now suppose $ a \in \roots $. By Proposition 15 Ch VI \S1 no. 5 \cite{bourbakiII}, there is $ w \in \finiteweylgroup $ such that $ w ( a ) \in \simpleroots $. Since the norm is $ \finiteweylgroup $-invariant, $ || a || = || w ( a ) || = L $. \textbf{(\ref{Lsimplylaced_cartansymmetric}) $ \boldsymbol{\Rightarrow} $ (\ref{Lsimplylaced_nomultbonds}):} By (\ref{Lsimplylaced_cartansymmetric}), $ \langle a, b^{\vee} \rangle = \langle b, a^{\vee} \rangle $ for all $ a, b \in \simpleroots $. If, further, $ a \neq b $ then $ a \notin \R \cdot b $ and so the table in Ch VI \S1 no. 3 \cite{bourbakiII} shows that $ -1, 0, 1 $ are the only possible values for $ \langle a, b^{\vee} \rangle $.
\end{proof}

\begin{lemma} \label{LcoeffRSdescription}
If $ \beta > 0 $ then $ w ( \beta ) < 0 $ if and only if $ \beta_R = 1 $. In fact,
\begin{enumerate}
\item $ \{ \beta > 0 \suchthat w ( \beta ) > 0, w^2 ( \beta ) > 0 \} = \{ \beta > 0 \suchthat \beta_R = 0, \beta_S = 0 \} $ \label{Esubset00}

\item $ \{ \beta > 0 \suchthat w ( \beta ) > 0, w^2 ( \beta ) < 0 \} = \{ \beta > 0 \suchthat \beta_R = 0, \beta_S = 1 \} $ \label{Esubset01}

\item $ \{ \beta > 0 \suchthat w ( \beta ) < 0, w^2 ( \beta ) > 0 \} = \{ \beta > 0 \suchthat \beta_R = 1, \beta_S = 0 \} $ \label{Esubset10}

\item $ \{ \beta > 0 \suchthat w ( \beta ) < 0, w^2 ( \beta ) < 0 \} = \{ \beta > 0 \suchthat \beta_R = 1, \beta_S = 1 \} $ \label{Esubset11}
\end{enumerate}
\end{lemma}

\begin{proof}
If $ \beta > 0 $ then Corollary \ref{CcoeffsRSonly01} says that $ \beta_R \in \{ 0, 1 \} $. Since $ w ( \simpleaffinegradients \setminus \{ R \} ) \subset \simpleroots $, it is obvious that if $ \beta_R = 0 $ then $ w ( \beta ) > 0 $ still. Conversely, if $ \beta_R = 1 $ then it is clear from the definition of ``highest'' and the fact that $ w ( R ) = - \highestroot $ that $ w ( \beta ) < 0 $. This proves the first claim. It is also clear from this that the validity of (\ref{Esubset00}) and (\ref{Esubset01}) will follow if it is proved, when $ \beta_R = 0 $, that $ \beta_S = 1 $ if and only if $ w^2 ( \beta ) < 0 $. But this is clear from the first claim again, since if $ \beta_R = 0 $ then $ \beta_S = 1 $ if and only if $ w ( \beta )_R = 1 $. The validity of (\ref{Esubset10}) and (\ref{Esubset11}) will follow if it is proved, when $ \beta_R = 1 $, that $ \beta_S = 1 $ if and only if $ w^2 ( \beta ) < 0 $. First, notice that when $ \beta_R = 1 $, due to cancellation, $ w ( \beta )_R = -1 $ if $ \beta_S = 0 $ and $ w ( \beta )_R = 0 $ if $ \beta_S = 1 $. By the same reasoning as in the beginning of the proof, if $ b < 0 $ then $ w ( b ) > 0 $ if and only if $ b_R = -1 $. Thus, if $ \beta_R = 1 $ then $ w ( w ( \beta ) ) > 0 $ if and only if $ \beta_S = 0 $.
\end{proof}

\begin{notation}
\begin{align*}
\flippingset_{0,0} \defeq \{ \beta > 0 \suchthat \beta_R = 0, \beta_S = 0 \} \\
\flippingset_{0,1} \defeq \{ \beta > 0 \suchthat \beta_R = 0, \beta_S = 1 \} \\
\flippingset_{1,0} \defeq \{ \beta > 0 \suchthat \beta_R = 1, \beta_S = 0 \} \\
\flippingset_{1,1} \defeq \{ \beta > 0 \suchthat \beta_R = 1, \beta_S = 1 \}
\end{align*}
\end{notation}

Note that, by Lemma \ref{LcoeffRSdescription},
\begin{itemize}
\item $ \{ \flippingset_{0,0}, \flippingset_{0,1}, \flippingset_{1,0}, \flippingset_{1,1} \} $ is a partition of $ \positiveroots $,
\item $ \lengthset ( w ) = \flippingset_{1,0} \sqcup \flippingset_{1,1} $, and
\item $ \flippingset ( w ) = \flippingset_{1,0} $
\end{itemize}

It is also important to observe that each of these sets has an extremal element: 
\begin{itemize}
\item $ S $ is the unique minimal element in $ \flippingset_{0,1} $,

\item $ R $ is the unique minimal element in $ \flippingset_{1,0} $, and 

\item $ \highestroot $ is the unique maximal element in $ \flippingset_{1,1} $.
\end{itemize}

\begin{example}
If $ G $ is the adjoint group D5 and $ w $ is as in Example \ref{EXfirstD5example} then
\begin{align*}
\flippingset_{0,0} &= \left\{ \dynkinDfive{0}{1}{0}{0}{0}, \dynkinDfive{0}{0}{1}{0}{0}, \dynkinDfive{0}{0}{0}{0}{1}, \dynkinDfive{0}{1}{1}{0}{0}, \dynkinDfive{0}{0}{1}{0}{1}, \dynkinDfive{0}{1}{1}{0}{1} \right\} \\
\flippingset_{0,1} &= \left\{ \dynkinDfive{1}{0}{0}{0}{0}, \dynkinDfive{1}{1}{0}{0}{0}, \dynkinDfive{1}{1}{1}{0}{0}, \dynkinDfive{1}{1}{1}{0}{1} \right\} \\
\flippingset_{1,0} &= \left\{ \dynkinDfive{0}{0}{0}{1}{0}, \dynkinDfive{0}{0}{1}{1}{0}, \dynkinDfive{0}{1}{1}{1}{0}, \dynkinDfive{0}{0}{1}{1}{1}, \dynkinDfive{0}{1}{1}{1}{1}, \dynkinDfive{0}{1}{2}{1}{1} \right\} \\
\flippingset_{1,1} &= \left\{ \dynkinDfive{1}{1}{1}{1}{0}, \dynkinDfive{1}{1}{1}{1}{1}, \dynkinDfive{1}{1}{2}{1}{1}, \dynkinDfive{1}{2}{2}{1}{1} \right\}
\end{align*}
\end{example}

\begin{defn}
Define 
\begin{equation*}
\Sigma_{R,S} \subset \flippingset_{0,1} \times \flippingset_{1,0} \times \flippingset_{1,1} 
\end{equation*}
to be the subset of all $ ( \alpha, \beta, \gamma ) $ such that $ \alpha + \beta = \gamma $. Define $ \pr_{0,1}, \pr_{1,0}, \pr_{1,1} $ to be the projections 
\begin{align*}
\pr_{0,1} : \Sigma_{R,S} \longrightarrow \flippingset_{0,1} \\
\pr_{1,0} : \Sigma_{R,S} \longrightarrow \flippingset_{1,0} \\
\pr_{1,1} : \Sigma_{R,S} \longrightarrow \flippingset_{1,1}
\end{align*}
\end{defn}

\begin{prop}[Constant Fibers] \label{Pconstantfibers}
All fibers in $ \Sigma_{R,S} $ of $ \pr_{0,1} $ are the same size, all fibers in $ \Sigma_{R,S} $ of $ \pr_{1,0} $ are the same size, and all fibers in $ \Sigma_{R,S} $ of $ \pr_{1,1} $ are the same size.
\end{prop}

\begin{proof}
The proof of this will be completed in \S\ref{SSproofofconstantfibers}, after some corollaries are derived.
\end{proof}

\begin{example}
Let $ G $ be the adjoint group D5 and let $ w, R, S $ be as in Example \ref{EXfirstD5example}. The following table expresses $ \Sigma_{R,S} $:
\begin{center}
\begin{tabular}{c|cccc}
$ + $ & $ \dynkinDfive{1}{0}{0}{0}{0} $ & $\dynkinDfive{1}{1}{0}{0}{0} $ & $ \dynkinDfive{1}{1}{1}{0}{0} $ & $ \dynkinDfive{1}{1}{1}{0}{1} $ \\
\hline
$ \dynkinDfive{0}{0}{0}{1}{0} $ & $ \notin $ & $ \notin $ & $ \dynkinDfive{1}{1}{1}{1}{0} $ & $ \dynkinDfive{1}{1}{1}{1}{1} $ \\
$ \dynkinDfive{0}{0}{1}{1}{0} $ & $ \notin $ & $ \dynkinDfive{1}{1}{1}{1}{0} $ & $ \notin $ & $ \dynkinDfive{1}{1}{2}{1}{1} $ \\
$ \dynkinDfive{0}{1}{1}{1}{0} $ & $ \dynkinDfive{1}{1}{1}{1}{0} $ & $ \notin $ & $ \notin $ & $ \dynkinDfive{1}{2}{2}{1}{1} $ \\
$ \dynkinDfive{0}{0}{1}{1}{1} $ & $ \notin $ & $ \dynkinDfive{1}{1}{1}{1}{1} $ & $ \dynkinDfive{1}{1}{2}{1}{1} $ & $ \notin $ \\
$ \dynkinDfive{0}{1}{1}{1}{1} $ & $ \dynkinDfive{1}{1}{1}{1}{1} $ & $ \notin $ & $ \dynkinDfive{1}{2}{2}{1}{1} $ & $ \notin $ \\
$ \dynkinDfive{0}{1}{2}{1}{1} $ & $ \dynkinDfive{1}{1}{2}{1}{1} $ & $ \dynkinDfive{1}{2}{2}{1}{1} $ & $ \notin $ & $ \notin $
\end{tabular}
\end{center}
The sizes of the fibers are quickly deduced from this table. The fact that the fibers of $ \pr_{0,1} $ and of $ \pr_{1,1} $ are the same size (= 3), and that the three sizes sum to the Coxeter Number (= 8) is not a coincidence; see Theorem \ref{Tmainformula} below.
\end{example}

\begin{examplenum} \label{EXfirstE6example}
Let $ G $ be the adjoint group E6. Following Plate V \cite{bourbakiII}, enumerate $ \simpleroots = \{ \alpha_1, \alpha_2, \alpha_3, \alpha_4, \alpha_5, \alpha_6 \} $ and depict any $ m_1 \alpha_1 + m_2 \alpha_2 + m_3 \alpha_3 + m_4 \alpha_4 + m_5 \alpha_5 + m_6 \alpha_6 \in \roots $ in ``Dynkin format'' as $ \dynkinEsix{m_1}{m_3}{m_4}{m_5}{m_6}{m_2} $. By Plate V (XII) \cite{bourbakiII}, there is a generator $ w \in \prW ( \Omega ) \cong \Z / 3 \Z $ whose permutation of $ \simpleaffinegradients $ has disjoint cycle decomposition $ ( \alpha_1, - \highestroot, \alpha_6 ) ( \alpha_3, \alpha_2, \alpha_5 ) ( \alpha_4 ) $, so $ R = \dynkinEsix{1}{0}{0}{0}{0}{0} $ and $ S = \dynkinEsix{0}{0}{0}{0}{1}{0} $. By computing a table\footnote{The table for E6 can be viewed by editing the \texttt{.tex} file and rebuilding the document.} as in the previous example, it is seen that the fibers of $ \pr_{0,1} $, $ \pr_{1,0} $, $ \pr_{1,1} $ are all the same size (= 4).
\end{examplenum}

\begin{notation}
$ F_{i,j} \defeq \sum_{ a \in \flippingset_{i,j} } a^{\vee} $
\end{notation}

\begin{corollary} \label{Cdependencerelation}
If $ a, b, c \in \N $ are the sizes of the fibers of $ \pr_{0,1}, \pr_{1,0}, \pr_{1,1} $ then 
\begin{equation*}
a F_{0,1} + b F_{1,0} = c F_{1,1}
\end{equation*}
\end{corollary}

\begin{proof}
Because fibers of $ \pr_{1,1} $ are all the same size $ c $, summing all $ \alpha $ and all $ \beta $ for which $ ( \alpha, \beta, \alpha + \beta ) \in \Sigma_{R,S} $ clearly yields $ c F_{1,1} $. On the other hand, because fibers of $ \pr_{0,1} $ are the same size $ a $ and the fibers of $ \pr_{1,0} $ are the same size $ b $, this sum is also $ a F_{0,1} + b F_{1,0} $.
\end{proof}

The height function $ \height : \roots \rightarrow \Z $, and pullbacks of it by any element of $ \finiteweylgroup $, extend to become elements of $ V^{*} $. In particular, $ \height, \height \circ w, \height \circ w^2 \in V^{*} $. It happens that 
\begin{equation*}
F_w \in \myspan ( \height, \height \circ w, \height \circ w^2 )
\end{equation*}
within $ V^{*} $; a much more precise version of this is next:

\begin{theorem}[Summation Formula \#2] \label{Tmainformula}
If $ a, b, c \in \N $ are as in Corollary \ref{Cdependencerelation} then $ a + b + c = h $ (Coxeter Number), $ a = c $, and 
\begin{equation*}
h \cdot F_w ( \emptyargument ) = c \cdot [ \height ( \emptyargument ) - 2 \cdot \height ( w ( \emptyargument ) ) + \height ( w^2 ( \emptyargument ) ) ]
\end{equation*}
\end{theorem}

\begin{proof}
Set $ \varrho \defeq \sum_{a > 0} a^{\vee} $ (so $ \varrho = 2 \hat{\rho} $, in previous notation) and recall that 
\begin{equation*}
2 \cdot \height = \langle \emptyargument, \varrho \rangle
\end{equation*}

Set $ \varrho^{\prime} \defeq \sum_{a \in \lengthset ( w )} a^{\vee} $ and conclude from the identity $ \height - \height \circ w = \langle \emptyargument, \varrho^{\prime} \rangle $ (Proposition \ref{Pfirstformula}) that 
\begin{equation*}
\langle \emptyargument, w^{-1} ( \varrho ) \rangle = 2 \cdot \height \circ w = \langle \emptyargument, \varrho - 2 \varrho^{\prime} \rangle
\end{equation*}

Finally, note that $ w^{-1} ( \lengthset ( w ) ) = \{ a \in \roots \suchthat w ( a ) > 0, w^2 ( a ) < 0 \} $. Set $ \varrho^{\prime \prime} \defeq \sum_{a \in w^{-1} ( \lengthset ( w ) )} a^{\vee} $, so that $ w^{-1} ( \varrho^{\prime} ) = \varrho^{\prime \prime} $ and therefore 
\begin{equation*}
\langle \emptyargument, w^{-2} ( \varrho ) \rangle = 2 \cdot \height \circ w^2 = \langle \emptyargument, \varrho - 2 \varrho^{\prime} - 2 \varrho^{\prime \prime} \rangle
\end{equation*}

Since the matrix sending $ \varrho, \varrho^{\prime}, \varrho^{\prime \prime} $ to $ \varrho, w^{-1} ( \varrho ), w^{-2} ( \varrho ) $ is triangular with non-zero diagonal, and since $ F_w = \langle \emptyargument, F_{1,0} \rangle $, it suffices to show $ F_{1,0} \in \myspan ( \varrho, \varrho^{\prime}, \varrho^{\prime \prime} ) $ and to determine a linear combination.

Recalling the notation $ F_{i,j} = \sum_{a \in \flippingset_{i,j}} a^{\vee} $ it is clear that 
\begin{align}
\nonumber \varrho &= F_{0,0} + F_{0,1} + F_{1,0} + F_{1,1} \\
\label{Evarrhoprime} \varrho^{\prime} &= F_{1,0} + F_{1,1} \\
\nonumber \varrho^{\prime \prime} &= F_{0,1} - F_{1,0}
\end{align}

By Corollary \ref{Cdependencerelation}, $ F_{1,1} = A \cdot F_{0,1} + B \cdot F_{1,0} $ for $ A = a / c $ and $ B = b / c $ and so $ \varrho^{\prime} = A \cdot F_{0,1} + ( B + 1 ) \cdot F_{1,0} $. Thus, $ \frac{1}{A} \cdot \varrho^{\prime} - \varrho^{\prime \prime} = ( \frac{B + 1}{A} + 1 ) \cdot F_{1,0} $ and so
\begin{equation*}
c \cdot \varrho^{\prime} - a \cdot \varrho^{\prime \prime} = ( a + b + c ) \cdot F_{1,0}
\end{equation*}

Using the identities $ \langle \emptyargument, \varrho^{\prime} \rangle = \height - \height \circ w $ and $ \langle \emptyargument, \varrho^{\prime \prime} \rangle = \height \circ w - \height \circ w^2 $ inverse to those above, a preliminary version of the desired linear combination is obtained:
\begin{equation} \label{Epreliminaryformula}
( a + b + c ) \cdot F_w = c \cdot \height - ( a + c ) \cdot \height \circ w + a \cdot \height \circ w^2
\end{equation}

Since $ S \in \flippingset_{0,1} $, the common size $ a $ of the fibers in $ \Sigma $ of $ \pr_{0,1} $ may be computed by asking how many $ \gamma \in \flippingset_{1,1} $ satisfy $ \gamma - S \in \roots $. Equivalently, how many $ \gamma \in \flippingset_{1,1} $ satisfy $ S - \gamma \in \roots $. I claim that this number is simply $ \langle S, F_{1,1} \rangle $. Since $ \roots $ is simply-laced, $ \langle S, \gamma^{\vee} \rangle \in \{ -1, 0, 1 \} $. If $ p, q \in \N $ are maximal such that $ S + i \gamma \in \roots $ for all $ - q \leq i \leq p $ then necessarily $ p = 0 $ since $ S + \gamma \notin \roots $ ($ S + \gamma $ has $ S $-coefficient $ 2 $, which contradicts Corollary \ref{CcoeffsRSonly01}) and $ q = 1 $ if and only if $ S - \gamma \in \roots $. Hence, $ a = \langle S, F_{1,1} \rangle $. By (\ref{Evarrhoprime}),
\begin{align*}
\langle S, \varrho^{\prime} \rangle &= \langle S, F_{1,0} \rangle + \langle S, F_{1,1} \rangle \\
\langle S, \varrho^{\prime} \rangle &= F_w ( S ) + a \\
\height ( S ) - \height ( w ( S ) ) &= \frac{1}{a+b+c} [ c \height ( S ) - ( a + c ) \height ( w ( S ) ) + a \height ( w^2 ( S ) ) ] + a \\
1 - 1 &= \frac{1}{a+b+c} [ c - ( a + c ) + a ( 1 - h ) ] + a \\
0 &= - a h + ( a + b + c ) a \\
h &= a + b + c
\end{align*}

Using the known formula (\ref{Epreliminaryformula}) in two different ways, compute:
\begin{align*}
F_w ( w^{-1} ( R ) ) = F_w ( S ) = \frac{1}{h} [ c \cdot 1 - ( a + c ) \cdot 1 + a \cdot ( 1 - h ) ] &= - a \\
F_{w^{-1}} ( w ( R ) ) = F_{w^{-1}} ( - \highestroot ) = \frac{1}{h} [ c \cdot ( 1 - h ) - ( a + c ) \cdot 1 + a \cdot 1 ] &= - c 
\end{align*}
By Lemma \ref{Lsymmetry}, $ F_w ( w^{-1} ( R ) ) = F_{w^{-1}} ( w ( R ) ) $, as desired.
\end{proof}

\begin{example}
Let $ G $ be the adjoint group D5 and let $ w $ be as in Example \ref{EXfirstD5example}. A formula for $ F_w $ may be computed by brute-force and doing so yields $ F_w ( \emptyargument ) = \frac{3}{8} \cdot \height ( \emptyargument ) - \frac{3}{4} \cdot \height ( w ( \emptyargument ) ) + \frac{3}{8} \cdot \height ( w^2 ( \emptyargument ) ) $. The Coxeter Number for D5 is $ h = 8 $ and, by the above table, $ ( a, b, c ) = ( 3, 2, 3 ) $. Thus, the predicted formula is $ 8 \cdot F_w ( \emptyargument ) = 3 \cdot [ \height ( \emptyargument ) - 2 \cdot \height ( w ( \emptyargument ) ) + \height ( w^2 ( \emptyargument ) ) ] $, which agrees.
\end{example}

\begin{example}
Let $ G $ be the adjoint group E6 and let $ w $ be as in Example \ref{EXfirstE6example}. A formula for $ F_w $ may be computed by brute-force and doing so yields $ F_w ( \emptyargument ) = \frac{1}{3} \cdot \height ( \emptyargument ) - \frac{2}{3} \cdot \height ( w ( \emptyargument ) ) + \frac{1}{3} \cdot \height ( w^2 ( \emptyargument ) ) $. The Coxeter Number for E6 is $ h = 12 $ and, by Example \ref{EXfirstE6example}, $ ( a, b, c ) = ( 4, 4, 4 ) $. Thus, the predicted formula is $ 12 \cdot F_w ( \emptyargument ) = 4 \cdot [ \height ( \emptyargument ) - 2 \cdot \height ( w ( \emptyargument ) ) + \height ( w^2 ( \emptyargument ) ) ] $, which agrees.
\end{example}

The following fact, which has apparent significance for the comparison of $ \canonicalsection ( w )^2 $ with $ \canonicalsection ( w^2 ) $, will be used later:

\begin{corollary} \label{Csumeven}
$ F_w ( R ) \in 2 \Z $
\end{corollary}

\begin{proof}
Suppose first that $ | w | = 2 $. In this case $ \flippingset ( w ) = \lengthset ( w ) $ and Formula \#1 (Proposition \ref{Pfirstformula}) says $ F_w ( R ) = 1 - ( 1 - h ) = h $. This proves the claim in that case since the Coxeter Number $ h $ is even in ``most'' cases, including all cases currently considered (this is less obvious for type A, but still true: $ | w | = 2 $ implies $ 2 $ divides $ \# \Omega_{\ad} = \rank ( \roots ) + 1 = h $). Now suppose $ | w | \geq 3 $. Formula \#2 (Theorem \ref{Tmainformula}) says $ h \cdot F_w ( R ) = c [ 1 - 2 ( 1 - h ) + 1 ] = 2 \cdot c \cdot h $, which implies the claim regardless of $ h $.
\end{proof}

\begin{example}
In \cite{roro}, it was necessary to calculate $ F_w ( R ) $. The brute-force calculation there for type D and rank $ \ell $ produced the number $ 2 \ell - 4 $. For D5 and $ w $ as in Example \ref{EXfirstD5example}, the result is $ 6 $. By the formula above, $ h \cdot F_w ( R ) = 3 \cdot [ \height ( R ) - 2 \height ( - \highestroot ) + \height ( - w ( \highestroot ) ) ] = 3 \cdot [ 1 - 2 ( 1 - h ) + 1 ] = 3 \cdot 2 \cdot h $, as expected.
\end{example}

\begin{remark}
The proof of Corollary \ref{Csumeven} explains the curious event that the result of the brute-force calculations in \cite{roro} was sometimes $ h $ (e.g. when $ G $ was adjoint E7).
\end{remark}

\subsection{Proof of Proposition \ref{Pconstantfibers}} \label{SSproofofconstantfibers}

To prove Proposition \ref{Pconstantfibers}, I construct for each $ i, j $ bijections among the fibers of $ \pr_{i,j} $, starting with $ \pr_{1,1} $.

\begin{lemma} \label{Linduction}
Let $ \gamma^{\prime}, \gamma \in \flippingset_{1,1} $ be arbitrary and set $ a \defeq \gamma^{\prime} - \gamma $. Assume that $ \gamma^{\prime} \geq \gamma $ and that $ a \in \roots $. \textbf{Assertion:}
\begin{enumerate}
\item If $ ( \alpha^{\prime}, \beta^{\prime}, \gamma^{\prime} ) \in \Sigma_{R,S} $ then exactly one of $ \langle \alpha^{\prime}, a^{\vee} \rangle $, $ \langle \beta^{\prime}, a^{\vee} \rangle $ is $ 1 $, the other is $ 0 $, and there is a canonical choice of $ \nu^{\prime} \in \{ \alpha^{\prime}, \beta^{\prime} \} $ such that $ \nu^{\prime} - a \in \roots $. \label{Linductionhighfibertolowfiber}

\item If $ ( \alpha, \beta, \gamma ) \in \Sigma_{R,S} $ then exactly one of $ \langle \alpha, a^{\vee} \rangle $, $ \langle \beta, a^{\vee} \rangle $ is $ -1 $, the other is $ 0 $, and there is a canonical choice of $ \nu \in \{ \alpha, \beta \} $ such that $ \nu + a \in \roots $. \label{Linductionlowfibertohighfiber}
\end{enumerate}
\end{lemma}

\begin{proof}
\textbf{(\ref{Linductionhighfibertolowfiber})} Since $ \alpha^{\prime}, \beta^{\prime} \neq a $, due to the fact that all three come from different $ \flippingset_{i,j} $, the last claim follows from the first two by Theorem 1(i) Ch VI \S1 no. 3 \cite{bourbakiII}: $ \langle \nu^{\prime}, a^{\vee} \rangle > 0 $ and $ \nu^{\prime} \neq a $ implies $ \nu^{\prime} - a \in \roots $. Thus, it suffices to prove the first two claims. Since $ \langle \gamma^{\prime}, a^{\vee} \rangle = \langle \alpha^{\prime}, a^{\vee} \rangle + \langle \beta^{\prime}, a^{\vee} \rangle $ and the only possible two-term partitions of $ 1 $ by $ \{ -1, 0, 1 \} $ are $ 1 = 1 + 0 $ and $ 1 = 0 + 1 $ (recall that $ \roots $ is simply-laced), it suffices to show merely that $ \langle \gamma^{\prime}, a^{\vee} \rangle = 1 $. Let $ p, q \in \N $ be maximal such that $ \gamma^{\prime} + i a \in \roots $ for all $ - q \leq i \leq p $ and set $ \gamma^{\prime \prime} := \gamma^{\prime} + p a $. By Proposition 9(iii) Ch VI \S1 no. 3 \cite{bourbakiII}, $ \langle \gamma^{\prime \prime}, a^{\vee} \rangle = p + q $. Since $ \roots $ is simply-laced, $ p + q \in \{ -1, 0, 1 \} $. Since $ \gamma^{\prime} - a = \gamma \in \roots $, $ q \geq 1 $. It follows from all this that $ ( p, q ) = ( 0, 1 ) $. By the same Proposition 9(iii) Ch VI \S1 no. 3 \cite{bourbakiII}, $ \langle \gamma^{\prime}, a^{\vee} \rangle = q - p = 1 $, as desired. \textbf{(\ref{Linductionlowfibertohighfiber})} The proof of this is nearly identical.
\end{proof}

Fix $ \gamma, \gamma^{\prime} \in \flippingset_{1,1} $, set $ a \defeq \gamma^{\prime} - \gamma $, and assume as before that $ a \in \positiveroots $. By Lemma \ref{Linduction}(\ref{Linductionlowfibertohighfiber}), if $ ( \alpha, \beta, \gamma ) \in \Sigma_{R,S} $ then it is true either that $ ( \alpha + a, \beta, \gamma^{\prime} ) \in \Sigma_{R,S} $ or that $ ( \alpha, \beta + a, \gamma^{\prime} ) \in \Sigma_{R,S} $ and there is a canonical choice among these two possibilities. 

In other words, if the fibers of $ \pr_{1,1} $ over $ \gamma $ and $ \gamma^{\prime} $ are denoted $ \mathfrak{f} $ and $ \mathfrak{f}^{\prime} $ then Lemma \ref{Linduction}(\ref{Linductionlowfibertohighfiber}) provides a function $ \mathfrak{f} \rightarrow \mathfrak{f}^{\prime} $. 

Similarly, Lemma \ref{Linduction}(\ref{Linductionhighfibertolowfiber}) provides a function $ \mathfrak{f}^{\prime} \rightarrow \mathfrak{f} $.

\begin{prop} \label{Pneighboringfiberssamesize}
As above, let $ \gamma, \gamma^{\prime} \in \flippingset_{1,1} $ be arbitrary such that $ \gamma^{\prime} - \gamma \in \positiveroots $, and denote by $ \mathfrak{f} $ and $ \mathfrak{f}^{\prime} $ the fibers of $ \pr_{1,1} $ over $ \gamma $ and $ \gamma^{\prime} $. \textbf{Assertion:} The functions $ \mathfrak{f} \rightarrow \mathfrak{f}^{\prime} $ and $ \mathfrak{f}^{\prime} \rightarrow \mathfrak{f} $ from above are inverse. In particular, $ \mathfrak{f} $ and $ \mathfrak{f}^{\prime} $ are the same size.
\end{prop}

\begin{proof}
Suppose $ ( \alpha^{\prime}, \beta^{\prime}, \gamma^{\prime} ) \in \Sigma_{R,S} $. Reviewing the proof of Lemma \ref{Linduction}, if $ \alpha^{\prime} $ happens to be the element $ \nu^{\prime} \in \{ \alpha^{\prime}, \beta^{\prime} \} $ such that $ \langle \nu^{\prime}, a^{\vee} \rangle = 1 $ then the claim is that $ \alpha := \alpha^{\prime} - a \in \roots $ is the element $ \nu \in \{ \alpha^{\prime} - a, \beta \} $ such that $ \langle \nu, a^{\vee} \rangle = -1 $. This is obvious by consideration of strings: the $ a $-string through $ \alpha^{\prime} $ is $ \alpha^{\prime}, \alpha^{\prime} - a $ and so the $ a $-string through $ \alpha $ must be $ \alpha + a, \alpha $ which means that $ \langle \alpha, a^{\vee} \rangle = -1 $. The proof when instead $ \beta^{\prime} $ is the element $ \nu^{\prime} \in \{ \alpha^{\prime}, \beta^{\prime} \} $ such that $ \langle \nu^{\prime}, a^{\vee} \rangle = 1 $ is identical. This establishes that $ \mathfrak{f} \rightarrow \mathfrak{f}^{\prime} $ is a left-inverse to $ \mathfrak{f}^{\prime} \rightarrow \mathfrak{f} $. That the other composite is the identity is proved in a nearly identical way.
\end{proof}

\begin{corollary} \label{Cfibersinbijection}
For any $ \gamma \in \flippingset_{1,1} $, the fiber of $ \pr_{1,1} $ over $ \gamma $ is the same size as the fiber of $ \pr_{1,1} $ over $ \highestroot $.
\end{corollary}

\begin{proof}
This follows by Induction on $ \height ( \gamma ) $, with Base Case $ \gamma = \highestroot $. For the Base Case, there is nothing to prove. Now, suppose $ \alpha_1, \ldots, \alpha_n \in \simpleroots $ (not necessarily distinct, $ n \geq 1 $) are such that $ \highestroot - \gamma = \alpha_1 + \cdots + \alpha_n $. Necessarily, $ \alpha_i \in \flippingset_{0,0} $ for all $ i $. By Proposition 19 Ch VI \S1 no. 6 \cite{bourbakiII}, it is possible to reorder the list $ \gamma, \alpha_1, \ldots, \alpha_n $ so that each partial sum is in $ \roots $. Fix such an ordering $ \beta_0, \beta_1, \ldots, \beta_n $. There are two cases: either $ \beta_0 = \gamma $ or $ \beta_0 \neq \gamma $. In the former case, $ \beta_0 + \beta_1 \in \roots $ means that there is $ i $ such that $ \gamma + \alpha_i \in \roots $. In the latter case, there is non-empty $ I \subset \{ 1, \ldots, n \} $ such that $ \sum_{i \in I} \alpha_i \in \roots $ and $ \sum_{i \in I} \alpha_i + \gamma \in \roots $. In either case, there is $ a \in \positiveroots $ such that $ \gamma^{\prime} := \gamma + a \in \roots $ and necessarily $ \height ( \gamma^{\prime} ) > \height ( \gamma ) $. Thus, the hypotheses of Proposition \ref{Pneighboringfiberssamesize} are satisfied for $ \gamma^{\prime}, \gamma $ and so the size of the fiber over $ \gamma $ is the same as the size of the fiber over $ \gamma^{\prime} $, which is the same as the size of the fiber over $ \highestroot $ by the Induction Hypothesis.
\end{proof}

Nearly identical arguments prove that all fibers of $ \pr_{0,1} $ are the same size and that all fibers of $ \pr_{1,0} $ are the same size. I sketch this argument next, starting with $ \pr_{0,1} $.

Suppose $ ( \alpha, \beta, \gamma ) \in \Sigma_{R,S} $ and $ \alpha^{\prime} \in \flippingset_{0,1} $ is such that $ \alpha^{\prime} - \alpha \in \positiveroots $. By an argument similar\footnote{It was never important that the two given elements $ \gamma, \gamma^{\prime} $ belonged to $ \flippingset_{1,1} $; it was important only that they belonged to the same $ \flippingset_{i,j} $ different from $ \flippingset_{0,0} $.} to Lemma \ref{Linduction}(\ref{Linductionhighfibertolowfiber}), there are canonical $ \beta^{\prime} \in \flippingset_{1,0} $ and $ \gamma^{\prime} \in \flippingset_{1,1} $ such that $ ( \alpha^{\prime}, \beta^{\prime}, \gamma^{\prime} ) \in \Sigma_{R,S} $. If the fibers of $ \pr_{0,1} $ over $ \alpha $ and $ \alpha^{\prime} $ are denoted $ \mathfrak{f} $ and $ \mathfrak{f}^{\prime} $ then this yields a function $ \mathfrak{f} \rightarrow \mathfrak{f}^{\prime} $ by $ ( \alpha, \beta, \gamma ) \mapsto ( \alpha^{\prime}, \beta^{\prime}, \gamma^{\prime} ) $. An argument similar to Lemma \ref{Linduction}(\ref{Linductionlowfibertohighfiber}) produces a function $ \mathfrak{f}^{\prime} \rightarrow \mathfrak{f} $. An argument similar to Proposition \ref{Pneighboringfiberssamesize} proves that these two functions are inverse, hence that $ \mathfrak{f}^{\prime} $ and $ \mathfrak{f} $ are the same size. For arbitrary $ \alpha \in \flippingset_{0,1} $, an Induction similar to that in Corollary \ref{Cfibersinbijection} shows that the fiber of $ \pr_{0,1} $ over any $ \alpha $ is the same size as the fiber of $ \pr_{0,1} $ over $ S $.

A nearly identical argument to that just given proves that all fibers of $ \pr_{1,0} $ are the same size.

\begin{remark}
The fact that $ \flippingset_{0,1}, \flippingset_{1,0}, \flippingset_{1,1} $ have extremal elements $ S, R, \highestroot $ is valuable precisely for inductions like that in the proof of Corollary \ref{Cfibersinbijection}.
\end{remark}

\section{Dependence Relations and Characters} \label{Scharacters}

Let $ D $ be a $ \Z $-linear combination of elements of $ \roots $; formally, 
\begin{equation*}
D = \{ ( m_1, a_1 ), \ldots, ( m_k, a_k ) \}
\end{equation*}
for distinct $ a_i \in \roots $ and various $ m_i \in \Z $. Given such a $ D $, one may use the scalars $ c ( n, a ) \in \field^{\times} $ from \S\ref{Scanonicalrepresentatives} to define a function 
\begin{align*}
c_D : N_G ( T ) ( \field ) &\longrightarrow \field^{\times} \\
n &\longmapsto c ( n, a_1 )^{m_1} \cdots c ( n, a_k )^{m_k}
\end{align*}

In such generality, this $ c_D $ is not a very good function and probably not worth considering anyway. However, with additional hypotheses that are satisfied in practice, an important function with good properties is gained:

\begin{defn}
Let $ D $ be as above. $ w \in \finiteweylgroup $ is said to ``fix'' $ D $ iff $ w $ stabilizes $ \{ a_1, \ldots, a_k \} $ and $ m_i = m_j $ whenever $ w ( a_i ) = a_j $. Such a $ D $ is called a ``dependence relation'' iff $ m_1 a_1 + \cdots + m_k a_k = 0 $.
\end{defn}
Clearly, the set of elements fixing a given $ D $ is a subgroup.

\begin{examplenum} \label{EXhighestrootrelationexample}
Let $ D $ be the relation $ m_0 \alpha_0 + m_1 \alpha_1 + \cdots + m_{\ell} \alpha_{\ell} = 0 $ for $ \simpleroots = \{ \alpha_1, \ldots, \alpha_{\ell} \} $, $ \alpha_0 := - \highestroot $, $ m_0 := 1 $, $ \highestroot = m_1 \alpha_1 + \cdots + m_{\ell} \alpha_{\ell} $. Since $ \prW ( \Omega ) \subset \finiteweylgroup $ is the subgroup permuting $ \{ \alpha_0, \alpha_1, \ldots, \alpha_{\ell} \} $, Lemma \ref{Lmultiplicitiespreserved} implies that $ \prW ( \Omega ) $ is the subgroup fixing $ D $. More generally, if $ D $ is a \emph{minimal} dependence relation and $ w \in \finiteweylgroup $ permutes the roots appearing in $ D $ then necessarily $ w $ fixes $ D $.
\end{examplenum}

\begin{lemma}[Restricts to Character]
Let $ D $ be a dependence relation and let $ \Gamma \subset \finiteweylgroup $ be the subgroup fixing $ D $. \textbf{Assertion:} $ c_D $ factors through $ \finiteweylgroup $ (as merely a function) and the restriction $ c_D \vert_{\Gamma} : \Gamma \rightarrow \field^{\times} $ is a group homomorphism.
\end{lemma}

\begin{proof}
If $ n \in N_G(T)(\field) $ and $ t \in T(\field) $ then $ c_D ( n t ) = c ( n t, a_1 )^{m_1} \cdots c ( n t, a_k )^{m_k} = c ( n, a_1 )^{m_1} \cdot a_1 ( t )^{m_1} \cdots c ( n, a_k )^{m_k} \cdot a_k ( t )^{m_k} = c_D ( n ) \cdot a_1 ( t )^{m_1} \cdots a_k ( t )^{m_k} $. Since $ D $ is a dependence relation, $ a_1 ( t )^{m_1} \cdots a_k ( t )^{m_k} = 1 $ and the first claim follows. Thus, given $ n \in N_G(T)(\field) $ representing $ w \in \finiteweylgroup $, denote by $ c_D ( w ) $ this common element. For the second claim, note that $ c ( w \cdot w^{\prime}, a ) = c ( w, a ) \cdot c ( w^{\prime}, w ( a ) ) $ for all $ w, w^{\prime} \in \finiteweylgroup $ and $ a \in \roots $. If $ w $ fixes $ D $ and $ w ( a_i ) = a_j $ then $ m_i = m_j $ and so $ c ( w \cdot w^{\prime}, a_i )^{m_i} = c ( w, a_i )^{m_i} \cdot c ( w^{\prime}, a_j )^{m_i} = c ( w, a_i )^{m_i} \cdot c ( w^{\prime}, a_j )^{m_j} $ and taking the product for all $ i $ yields the second claim.
\end{proof}

\begin{notation}
From now on, if $ D $ is a dependence relation then $ c_D $ refers to the character after restriction to the subgroup of $ \finiteweylgroup $ that fixes $ D $. The dependence relation from Example \ref{EXhighestrootrelationexample} will be called the ``Highest Root Relation''.
\end{notation}

Here is a thoroughly underwhelming example:
\begin{example}
If $ a \in \roots $ then it is standard\footnote{Lemma 9.2.2(ii) \cite{springer}} that $ c ( \canonicalsection ( s ), a ) \cdot c ( \canonicalsection ( s ), -a ) = 1 $ for any $ s \in \simpleroots $. Since $ c ( n n^{\prime}, a ) = c ( n, w^{\prime} ( a ) ) \cdot c ( n^{\prime}, a ) $ for $ n, n^{\prime} \in N_G(T)(\field) $ representing $ w, w^{\prime} \in \finiteweylgroup $, applying the previous fact repeatedly yields that if $ D $ is the dependence relation $ - a + a = 0 $ then $ c_D $ is trivial.
\end{example}

A more substantive example, which will be critical for the application to simple supercuspidals, is the following:
\begin{prop} \label{Ptrivialcharacter}
If $ D $ is the Highest Root Relation then $ c_D : \prW ( \Omega ) \rightarrow \field^{\times} $ is the trivial character.
\end{prop}

\begin{proof}
Fix $ w \in \prW ( \Omega ) $. I may assume that the realization of $ \roots $ in $ G $ is a Chevalley Realization, in which case $ c ( \canonicalsection ( w ), a ) = 1 $ whenever $ a, w(a) \in \simpleroots $. By this and Lemma \ref{Lmultiplicitiespreserved}, $ c_D ( w ) = c ( \canonicalsection ( w ), R ) \cdot c ( \canonicalsection ( w ), - \highestroot ) $. Since $ R, w^2 ( R ) \in \simpleroots $, $ c ( \canonicalsection ( w^2 ), R ) = 1 $ and so this is the scalar by which $ \canonicalsection ( w )^2 $ acts on $ U_R $. On the other hand, Proposition \ref{Pgeneralformulafortwococycle} says that this scalar is $ ( -1 )^{F_w(R)} $. Corollary \ref{Csumeven} says that $ F_w(R) $ is even, so $ c_D ( w ) = 1 $ as desired.
\end{proof}

Here is a small example for which the character is non-trivial:
\begin{example}
Let $ G $ be any (almost-)simple group of type B2, with $ \simpleroots = \{ \alpha, \beta \} $ and $ \positiveroots = \{ \alpha, \beta, \alpha + \beta, 2 \alpha + \beta \} $. Abbreviate $ \gamma := \alpha + \beta $ and $ \highestroot := 2 \alpha + \beta $. Let $ s $ be the reflection corresponding to $ \beta $. Use the realization of $ \roots $ from \S33.4 \cite{humphreysAG} and let $ n $ be the canonical representative of $ s $. Let $ D $ be the dependence relation $ \alpha + \gamma - \highestroot = 0 $ and note that $ s $ fixes $ D $, since it exchanges $ \alpha, \gamma $ and fixes $ \highestroot $. However, $ c_D ( s ) = c ( n, \alpha ) \cdot c ( n, \gamma ) \cdot c ( n, \highestroot )^{-1} = -1 $, since the data in Proposition \S33.4 \cite{humphreysAG} says that $ c ( n, \alpha ) = 1 $, $ c ( n, \gamma ) = -1 $, $ c ( n, \highestroot ) = 1 $.
\end{example}

\begin{remark}
The question of whether or not $ c_D $ is trivial seems to be connected with the question of whether or not there are stable points in vector spaces coming from the Moy-Prasad filtrations at points in the Bruhat-Tits building. In more detail, choosing any point in the Bruhat-Tits building provides a certain $ \resfield $-vector space $ \mathbf{V} $ with an action by a certain $ \resfield $-group. There is a notion of ``stability'' relative to this action, and if the point in the building is chosen to be the barycenter of a facet, then $ \mathbf{V} $ may potentially contain stable vectors. This vector space $ \mathbf{V} $ has a natural basis corresponding to certain elements of $ \roots $, these basis roots are permuted (more or less) by certain elements $ w $ of $ \finiteweylgroup $ which stabilize the facet, and various dependence relations $ D $ hold among those basis roots. For some questions, it is necessary to ask if $ c_D ( w ) = 1 $ for $ w $ fixing one of these $ D $; one example of such a question, with an affirmative answer, is given below as Theorem \ref{Tfixingrepresentative}. It seems that triviality of these various $ c_D $ is equivalent to, or at least implied by, the existence of stable vectors in $ \mathbf{V} $. This more general question will be pursued in subsequent work.
\end{remark}

\section{Application: ``Simple Supercuspidals''} \label{Saffgencharapp}

\subsection{Additional notation} \label{SSadditionalnotation}

Let $ \integers \subset \field $ be the valuation ring. Fix, once and for all, a uniformizer $ \uniformizer \in \integers $ and denote by $ \resfield \defeq \integers / ( \uniformizer ) $ the residue field.

Since $ G $ is assumed $ \field $-split, the origin of the apartment for $ T_{\ad} $ is hyperspecial and the connected parahoric $ \integers $-group $ \integralmodel{G} $ attached by Bruhat-Tits to this origin (a Chevalley group scheme) has a good special fiber: 
\begin{center}
\emph{The affine algebraic $ \resfield $-group $ \specialfiber{G} \defeq \integralmodel{G} \otimes_{\integers} \resfield $ is connected \\and reductive (cf. \S1.10.2 and \S3.8.1 of \cite{tits}).}
\end{center}
The $ \integers $-model $ \integralmodel{G} $ contains an $ \integers $-model $ \integralmodel{T} $ of the torus $ T $, its image $ \specialfiber{T} $ in $ \specialfiber{G} $ is a split maximal $ \resfield $-torus, and the root system of $ \specialfiber{G} $ relative to $ \specialfiber{T} $ may be identified with $ \roots $ (cf. \S3.5 and \S3.5.1 of \cite{tits}).

\preamble{Further references and details for the following material can be found in \cite{rorobern}.}

Let $ \kappa : G ( \field ) \twoheadrightarrow \kottwitzhomcodomain $ be the Kottwitz homomorphism. Let $ \kappa_{\ad} : G_{\ad} ( \field ) \twoheadrightarrow \kottwitzhomcodomain_{\ad} $ be the Kottwitz homomorphism for $ G_{\ad} $, and recall that $ \kottwitzhomcodomain_{\ad} $ is always a finite abelian group. There is a canonical homomorphism $ \kottwitzhomcodomain \rightarrow \kottwitzhomcodomain_{\ad} $ which is compatible, via $ \kappa $ and $ \kappa_{\ad} $, with the quotient map $ G ( \field ) \rightarrow G_{\ad} ( \field ) $ in the obvious way. The image of $ \omega $ in $ \kottwitzhomcodomain_{\ad} $ is denoted $ \omega_{\ad} $.

Set $ G_1 \defeq \kernel ( \kappa ) $ and let $ G^1 $ be the kernel of the composition $ G ( K ) \stackrel{\kappa}{\longrightarrow} \kottwitzhomcodomain \rightarrow \kottwitzhomcodomain / \kottwitzhomcodomain_{\tor} $, so that $ G_1 \subset G^1 $. It happens that $ \compacttorus = G_1 \cap T ( \field ) $. Also set $ \compactcenter \defeq G_1 \cap Z ( G ) ( \field ) $, and note that $ \compactcenter $ may be disconnected. Denote by $ \iwahorisubgroup \subset G ( \field ) $ the Iwahori subgroup associated to the alcove $ \alcove $, i.e. $ \iwahorisubgroup \defeq G_1 \cap \{ g \in G ( \field ) \suchthat g \cdot x = x \text{ for all } x \in \alcove \} $. Let $ \prounipotentiwahori \subset \iwahorisubgroup $ be the pro-unipotent radical. Set $ \onemodptorus \defeq T ( \field ) \cap \prounipotentiwahori $.

Since $ \compacttorus \subset G_1 $, the restriction of $ \kappa $ to $ N_G(T)(\field) $ descends to a group homomorphism $ W \rightarrow \kottwitzhomcodomain $. It is a fact that the restriction of this group homomorphism to $ \Omega $ is an isomorphism onto $ \kottwitzhomcodomain $. A useful corollary is that, after identifying $ \Omega $ with $ \kottwitzhomcodomain $ in this way, the restriction $ F \vert_{\Omega} $ of any section $ F $ of the canonical map $ N_G(T)(\field) \rightarrow W $ is a section of $ \kappa $.

\subsection{Generalities on affine generic characters}

\preamble{For more details of the material in this subsection, see \cite{GR} and \cite{RY}.}

For any point in the Bruhat-Tits building, which is taken to be $ \barycenter $ here, there is a Moy-Prasad Filtration $ J_0 \supset J_1 \supset J_2 \supset \cdots $. I omit many details of this filtration, but record that
\begin{itemize}
\item $ J_0 $ is the Iwahori subgroup $ \iwahorisubgroup $.
\item $ J_1 $ is the pro-unipotent radical $ \prounipotentiwahori $ of $ \iwahorisubgroup $.
\item $ J_0 / J_1 $ is the $ \resfield $-points of a connected affine algebraic $ \resfield $-group.
\item $ \mathbf{V}_{\barycenter} \defeq J_1 / J_2 $ is a finite-dimensional $ \resfield $-vector space and, via conjugation, admits an algebraic representation by the above $ \resfield $-group.
\end{itemize}

Conjugation of $ \prounipotentiwahori $ by $ \compacttorus $ induces a diagonalizable representation of $ \mathbf{V}_{\barycenter} $, and 
\begin{equation*}
\mathbf{V}_{\barycenter} = \bigoplus_{\alpha \in \simpleaffinegradients} \mathbf{V}_{\barycenter} ( \alpha )
\end{equation*}
where $ t \in \compacttorus $ acts on the $ 1 $-dimensional subspace $ \mathbf{V}_{\barycenter} ( \alpha ) $ by $ \alpha ( t ) $.

Let
\begin{equation*}
\lambda : \mathbf{V}_{\barycenter} \longrightarrow \resfield
\end{equation*}
be a stable $ \resfield $-linear functional, in the sense of \cite{RY}. Since $ \barycenter $ is the barycenter of an alcove, this notion reduces essentially to the notion of ``affine generic character'' of \cite{GR}:

\begin{center}
\emph{For each $ \alpha \in \simpleaffinegradients $, the restriction of $ \lambda $ to the line $ \mathbf{V}_{\barycenter} ( \alpha ) $ is \emph{non-zero}.}
\end{center}

\begin{notation}
$ G_{\barycenter} \defeq \{ g \in G ( \field ) \suchthat g \cdot \barycenter = \barycenter \} $
\end{notation}

Consider $ \lambda $ also to be a group homomorphism $ \prounipotentiwahori \rightarrow \resfield $ by pullback. If $ g \in G_{\barycenter} $ then conjugation by $ g $ induces an operator on $ \mathbf{V}_{\barycenter} $, still denoted by $ v \mapsto g \cdot v \cdot g^{-1} $.

\begin{defn}
$ g \cdot \lambda $ is the $ \resfield $-linear functional $ v \mapsto \lambda ( g \cdot v \cdot g^{-1} ) $ and 
\begin{equation*}
\fixer_{G_{\barycenter}} ( \lambda ) \defeq \{ g \in G_{\barycenter} \suchthat g \cdot \lambda = \lambda \}
\end{equation*}
\end{defn}

Note that $ Z ( G ) ( \field ) \subset \fixer_{G_{\barycenter}} ( \lambda ) $. Here is an easy description of $ G_{\barycenter} $ which is surely well-known and will be used later (for Proposition \ref{Pproductdecompofcharacterfixer}):

\begin{lemma} \label{Lbasicdecompofpointfixer}
If $ f : \kottwitzhomcodomain \hookrightarrow G ( \field ) $ is an arbitrary section of $ \kappa $ then $ G_{\barycenter} = f ( \kottwitzhomcodomain ) \cdot \iwahorisubgroup $.
\end{lemma}

Recall from \S\ref{SSadditionalnotation} that one may identify $ \kottwitzhomcodomain \directedisom \Omega $ and that if $ F $ is any section of the canonical map $ N_G(T)(\field) \rightarrow W $ then $ F \vert_{\Omega} $ is identified with a section of $ \kappa $.

\begin{proof}
Since $ \kappa : G / G_1 \directedisom \kottwitzhomcodomain $ and $ f $ is a section of $ \kappa $, $ f ( \kottwitzhomcodomain ) $ is a system of representatives for all $ G_1 $-cosets. Thus, for any $ g \in G_{\barycenter} $, there exists $ \omega \in \kottwitzhomcodomain $ and $ x \in G_1 $ such that $ g = f ( \omega ) \cdot x $. Since $ f ( \omega ) \in G_{\barycenter} $, $ x \in G_{\barycenter} $ also. Together, $ x \in G_1 \cap G_{\barycenter} $. It is a fact that $ \iwahorisubgroup = G_1 \cap G_{\barycenter} $ and so $ G_{\barycenter} \subset f ( \kottwitzhomcodomain ) \cdot \iwahorisubgroup $. The reverse inclusion is trivial.
\end{proof}

\subsection{Non-trivial elements in the fixer: adjoint case} \label{SSgoodrepresentativesadjointcase}

\begin{notation}
$ G_{\sss} \defeq G / Z ( G )^{\circ} $, the ``semisimplification'' of $ G $.
\end{notation}

\begin{theorem} \label{Tfixingrepresentative}
If $ G_{\sss} $ is adjoint\footnote{Note that this is equivalent to the hypothesis that $ Z ( G ) $ is connected.} then there exists a section $ \goodsectionext $ of the canonical map $ N_G(T)(\field) \rightarrow W $ such that $ \goodsectionext ( \Omega ) \subset \fixer_{G_{\barycenter}} ( \lambda ) $.
\end{theorem}

\begin{proof}
Fix arbitrary $ \omega \in \Omega $ and set $ \sigma := \prW ( \omega ) $. Arbitrarily order the weight basis for $ \mathbf{V}_{\barycenter} $, express $ \lambda $ as a $ \resfield $-matrix\footnote{Non-canonical identifications $ \mathbf{V}_{\barycenter} ( \alpha ) \cong \resfield $, which are always of the form $ ( \uniformizer ) / ( \uniformizer )^2 \cong \resfield $, are to be made using the fixed uniformizer $ \uniformizer $.} relative to this basis, and let $ \lambda_0, \lambda_1, \ldots, \lambda_{\ell} \in \integers $ be representatives for the entries of this matrix. By convention, $ \lambda_0 $ is the multiplier for the line $ \mathbf{V}_{\barycenter} ( -\highestroot ) $, and it will be convenient to set $ \alpha_0 := - \highestroot $. The assumption that $ \lambda $ is an affine generic character means that $ \lambda_0, \lambda_1, \ldots, \lambda_{\ell} \in \integers^{\times} $. 

Let $ n \in N_G ( T )( \field ) $ be an arbitrary representative of $ \omega $. By embedding $ \cochargroup ( T ) \hookrightarrow T ( \field ) $ via $ \mu \mapsto \mu ( \uniformizer )^{-1} $, there are $ n_{\finiteweylsubscript} \in N_G(T)(\field) $ representing $ \sigma $ and $ \tau \in \compacttorus $ such that the $ i $th entry of $ n \cdot \lambda $ is represented by $ \lambda_{ \sigma ( i ) } \cdot c ( n_{\finiteweylsubscript}, \alpha_i ) \cdot \alpha_i ( \tau ) $, where ``$ \sigma ( i ) $'' is shorthand for ``the $ j $ such that $ \sigma ( \alpha_i ) = \alpha_j $''. On the other hand, if $ t \in T ( \field ) $ then the $ i $th entry of $ t \cdot \lambda $ is $ \alpha_i ( t ) \cdot \lambda_i $. Thus, it suffices to show that there is $ t \in \compacttorus $ such that 
\begin{align}
\lambda_0 \cdot \alpha_0 ( t ) &= \lambda_{ \sigma ( 0 ) } \cdot c ( n_{\finiteweylsubscript}, \alpha_0 ) \label{Efixersystem} \\
\lambda_1 \cdot \alpha_1 ( t ) &= \lambda_{ \sigma ( 1 ) } \cdot c ( n_{\finiteweylsubscript}, \alpha_1 ) \nonumber \\
&{\vdots} \nonumber \\
\lambda_{\ell} \cdot \alpha_{\ell} ( t ) &= \lambda_{ \sigma ( \ell ) } \cdot c ( n_{\finiteweylsubscript}, \alpha_{\ell} ) \nonumber
\end{align}
because then $ n $ may be replaced by $ n \cdot \tau^{-1} \cdot t^{-1} $ to yield another representative of the same $ \omega $ with the desired behavior. \emph{Strictly speaking, it is required only that (\ref{Efixersystem}) is true after passing to $ \resfield $.}

Observe first that this seemingly overdetermined system (\ref{Efixersystem}) has a redundancy. The product of all equations in the system (\ref{Efixersystem}), each with multiplicity $ m_i $, yields $ \prod_i \lambda_i^{m_i} \cdot \prod_i \alpha_i ( t )^{m_i} = \prod_i \lambda_{ \sigma ( i ) }^{m_i} \cdot c_D ( \sigma ) $ where $ D $ is the Highest Root Relation. By Lemma \ref{Lmultiplicitiespreserved}, this reduces to $ \prod_i \alpha_i ( t )^{m_i} = c_D ( \sigma ) $. Since $ - \alpha_0 = m_1 \alpha_1 + \cdots + m_{\ell} \alpha_{\ell} $, this reduces to $ 1 = c_D ( \sigma ) $. Thus, to show that the equation $ \lambda_0 \cdot \alpha_0 ( t ) = \lambda_{ \sigma ( 0 ) } \cdot c ( n_{\finiteweylsubscript}, \alpha_0 ) $ is implied by the other equations, it suffices to show that $ 1 = c_D ( \sigma ) $ is true, which is Proposition \ref{Ptrivialcharacter}.

So, the claim reduces to the following: there exists $ t \in \compacttorus $ such that $ \alpha_i ( t ) = \lambda_i^{-1} \cdot \lambda_{ \sigma ( i ) } \cdot c ( n_{\finiteweylsubscript}, \alpha_i ) $ for all $ i = 1, \ldots, \ell $. By passing to the semisimplification $ G_{\sss} $ and using surjectivity of $ T \rightarrow T / Z ( G )^{\circ} $ on $ \resfield $-points (Lang's Theorem), I may assume that $ G = G_{\sss} $, which is adjoint by hypothesis. In this case, $ \compacttorus \rightarrow \myhom ( Q, \integers^{\times} ) $ is surjective due to $ Q = \chargroup ( T ) $, which means that such $ t $ exists.
\end{proof}

\textbf{Porism:} Note that, in the proof of Theorem \ref{Tfixingrepresentative}, the solution $ t \in \compacttorus = \integralmodel{T} ( \integers ) $ has a unique image $ \boldsymbol{t} \in \integralmodel{T} ( \resfield ) = \specialfiber{T} ( \resfield ) $.

\begin{remark}
One case of Theorem \ref{Tfixingrepresentative}, and the refinements in \S\ref{SSrefinements} below, was established in \cite{KL} for $ G = \GL $.
\end{remark}

\subsection{Non-trivial elements in the fixer: non-adjoint case} \label{SSgoodrepresentativesnonadjointcase}

In order to describe what happens when $ G_{\sss} $ is non-adjoint, it is convenient to introduce a cohomological tool to decide when a solution lifts from $ G_{\ad} $ to $ G_{\sss} $.

\begin{notation}
Recall that $ \resfield $ is perfect and, for any $ \resfield $-group $ H $, denote by $ H^1 ( \resfield, H ) $ the first Galois Cohomology $ H^1 ( \gal ( \overline{\resfield} / \resfield ), H ( \overline{\resfield} ) ) $.
\end{notation}

Let $ \specialfiber{T} $ be a split $ \resfield $-torus and fix $ n \in \N $. If $ \chi \in \chargroup ( \mathbf{T} ) $ then, by pushforward, $ \chi $ defines a group homomorphism 
\begin{equation*}
\del_{\chi} : \specialfiber{T} ( \resfield ) \rightarrow \resfield^{\times} / ( \resfield^{\times} )^n
\end{equation*}

By the Kummer Isomorphism\footnote{Despite the fact that this is usually accompanied by the assumption that $ \gcd ( \mychar ( \resfield ), m ) = 1 $, it is true here without this assumption since $ \resfield $ is perfect. Presumably, that assumption is imposed to guarantee surjectivity on points valued in the \emph{separable} closure.}, 
\begin{equation*}
H^1 ( \resfield, \MU_n ) \cong \resfield^{\times} / ( \resfield^{\times} )^n
\end{equation*}
and so $ \chi \mapsto \del_{\chi} $ defines a group homomorphism 
\begin{equation}
\chargroup ( \specialfiber{T} ) \longrightarrow \myhom ( \specialfiber{T} ( \resfield ), H^1 ( \resfield, \MU_n ) ) \label{Einducedconnectingmap}
\end{equation}

\begin{lemma} \label{Lallconnectingmapscomefromcharacters}
Map (\ref{Einducedconnectingmap}) is surjective.
\end{lemma}

\begin{proof}
Let $ R $ be a commutative $ \resfield $-algebra and $ \Gamma $ an abelian group. As usual, $ \specialfiber{T} ( R ) \directedisom \cochargroup ( \specialfiber{T} ) \otimes_{\Z} R^{\times} $ and so $ \myhom ( \specialfiber{T} ( R ), \Gamma ) \directedisom \myhom ( \cochargroup ( \specialfiber{T} ) \otimes_{\Z} R^{\times}, \Gamma ) $. Assume further that $ \Gamma $ is a quotient of $ R^{\times} $ with a given quotient map $ R^{\times} \canarrow \Gamma $. If $ \chi \in \chargroup ( \specialfiber{T} ) $ then, by duality, $ \chi $ defines a group homomorphism $ \cochargroup ( \specialfiber{T} ) \otimes_{\Z} R^{\times} \rightarrow R^{\times} \canarrow \Gamma $. Hence, there is a group homomorphism $ \delta : \chargroup ( \specialfiber{T} ) \rightarrow \myhom ( \cochargroup ( \specialfiber{T} ) \otimes_{\Z} R^{\times}, \Gamma ) $. Now, assume further that $ R^{\times} $ is finite and cyclic, say $ R^{\times} = \Z / N \Z $, in which case $ \Gamma $ is also, say $ \Gamma = \Z / M \Z $ with $ M | N $ (abbreviate these to $ \Z_N $ and $ \Z_M $). I claim now that $ \myhom ( \cochargroup ( \specialfiber{T} ) \otimes_{\Z} R^{\times}, \Gamma ) $ is simply $ \chargroup ( \specialfiber{T} ) \otimes_{\Z} \Gamma $ and that the map $ \delta $ is the canonical map. This follows from some standard hom/tensor identities:
\begin{align*}
\myhom_{\Z} ( \cochargroup ( \specialfiber{T} ) \otimes_{\Z} \Z_N, \Z_M ) &{\cong} \myhom_{\Z_N} ( \cochargroup ( \specialfiber{T} ) \otimes_{\Z} \Z_N, \Z_M ) \\
&{\cong} \myhom_{\Z} ( \cochargroup ( \specialfiber{T} ), \Z_M ) \\
&{\cong} \myhom_{\Z} ( \cochargroup ( \specialfiber{T} ), \Z ) \otimes_{\Z} \Z_M \\
&{\cong} \chargroup ( \specialfiber{T} ) \otimes_{\Z} \Z_M
\end{align*}
This\footnote{Since $ \Z_M $ is not a flat $ \Z $-module, the third isomorphism requires the fact that $ \cochargroup ( \specialfiber{T} ) $ is free.} shows that, with the assumptions on $ R, \Gamma $ made earlier, the composite $ \chargroup ( \specialfiber{T} ) \rightarrow \myhom ( \specialfiber{T} ( R ), \Gamma ) $ is surjective. To conclude the proof, choose $ R := \resfield $, choose $ \Gamma := \resfield^{\times} / ( \resfield^{\times} )^n $, note that $ R, \Gamma $ satisfy the needed assumptions since $ \resfield $ is a finite field, and note that the composite is clearly the map $ \chi \mapsto \del_{\chi} $ above.
\end{proof}

\begin{corollary}[Connecting Map is a Character] \label{Cconnectingmapisaproductofroots}
Let $ \specialfiber{G} $ be a split connected (almost-)simple affine algebraic $ \resfield $-group. Set $ \specialfiber{G}_{\ad} \defeq \specialfiber{G} / Z ( \specialfiber{G} ) $ and assume that $ Z ( \specialfiber{G} ) \cong \MU_n $ for some $ n \in \N $. Let $ \del : \specialfiber{G}_{\ad} ( \resfield ) \rightarrow H^1 ( \resfield, \MU_n ) $ be the connecting map from the long-exact-sequence associated to $ 1 \rightarrow \MU_n \rightarrow \specialfiber{G} \rightarrow \specialfiber{G}_{\ad} \rightarrow 1 $. \textbf{Assertion:} If $ \specialfiber{T}_{\ad} \subset \specialfiber{G}_{\ad} $ is a split maximal torus then the restriction of $ \partial $ to $ \specialfiber{T}_{\ad} ( \resfield ) $ is an algebraic character, hence a $ \Z $-linear combination of simple roots.
\end{corollary}

\begin{proof}
This follows immediately from Lemma \ref{Lallconnectingmapscomefromcharacters}, noting for the last conclusion that $ \chargroup ( \specialfiber{T}_{\ad} ) $ is spanned by simple roots since $ \specialfiber{G}_{\ad} $ is adjoint.
\end{proof}

\begin{examplenum} \label{EXconnectingmapischaracter}
Let $ F $ be any perfect field and let $ G $ be the $ F $-group $ \GL_3 $, so $ G_{\ad} = \PGL_3 $. Let $ T \subset \GL_3 $ be the diagonal torus and $ T_{\ad} \subset \PGL_3 $ its image. Let $ \epsilon_1, \epsilon_2, \epsilon_3 : T \rightarrow \Gm $ be the obvious generators for $ \chargroup ( T ) $, so that $ \simpleroots = \{ \alpha, \beta \} $ for $ \alpha = \epsilon_1 - \epsilon_2 $ and $ \beta = \epsilon_2 - \epsilon_3 $. Intrinsically, the connecting map $ \del : \PGL_3 ( F ) \rightarrow F^{\times} / ( F^{\times} )^3 $ coming from $ 1 \rightarrow \MU_3 \rightarrow \SL_3 \rightarrow \PGL_3 \rightarrow 1 $ is induced by the determinant $ \det : \GL_3 ( F ) \rightarrow F^{\times} $. Therefore, if $ t \in T ( F ) $ then $ \det ( t ) = \epsilon_1 ( t ) \cdot \epsilon_2 ( t ) \cdot \epsilon_3 ( t ) = \alpha ( t ) \cdot \beta ( t )^2 \cdot \epsilon_3 ( t )^3 $ and so, since values of $ \del $ are modulo $ ( F^{\times} )^3 $, this means that $ \del \vert_{T_{\ad}} = \alpha + 2 \beta $.
\end{examplenum}

Now, suppose $ G_{\sss} $ is (possibly) non-adjoint and let $ T_{\sss} \subset G_{\sss} $ be the image of $ T $, a split maximal torus in $ G_{\sss} $. Necessarily $ Z ( G_{\sss} ) $ is cyclic, since $ \roots $ is irreducible. Apply Theorem \ref{Tfixingrepresentative} to $ G_{\ad} $ and get a solution $ \boldsymbol{t}_{\ad} \in T_{\ad} ( \resfield ) $ to system (\ref{Efixersystem}). By the uniqueness of this solution as a $ \resfield $-point, there is a solution $ \boldsymbol{t} \in T_{\sss} ( \resfield ) $ to system (\ref{Efixersystem}) if and only if $ \del ( \boldsymbol{t}_{\ad} ) = 0 $, where $ \del $ is the connecting map in the long-exact-sequence associated to $ 1 \rightarrow Z ( G_{\sss} ) \rightarrow  G_{\sss} \rightarrow G_{\ad} \rightarrow 1 $. By system (\ref{Efixersystem}) itself, $ \alpha_i ( \boldsymbol{t}_{\ad} ) $ is a function of $ \lambda $ and $ \omega $ for all $ i $. By Corollary \ref{Cconnectingmapisaproductofroots}, this means that $ \del ( \boldsymbol{t}_{\ad} ) \in H^1 ( \resfield, Z ( G_{\sss} ) ) $ is also a function of $ \lambda $ and $ \omega $.

\begin{defn}
Define $ \del [ \lambda, \omega ] \in H^1 ( \resfield, Z ( G_{\sss} ) ) $ to be the class determined by the solution $ \boldsymbol{t}_{\ad} \in T_{\ad} ( \resfield ) $ to system (\ref{Efixersystem}).
\end{defn}

\begin{remark}
The definition of $ \del [ \lambda, \omega ] $ depends critically on the fact that $ \lambda $ is stable.
\end{remark}

The following is clear from the discussion preceding the Definition of $ \del [ \lambda, \omega ] $, and is recorded here for completeness:

\begin{prop}[Cohomological Obstruction] \label{Pcohomologicaltest}
Assume that $ Z ( G_{\sss} ) $ is cyclic. Fix $ \omega \in \Omega $. \textbf{Assertion:} There exists a representative $ n \in N_G(T)(\field) $ of $ \omega $ for which $ n \cdot \lambda = \lambda $ if and only if $ \del [ \lambda, \omega ] = 0 $.
\end{prop}

Note that if $ G $ is \emph{not} within the scope of Theorem \ref{Tfixingrepresentative}, i.e. if $ G_{\sss} $ is non-adjoint, then $ G $ is necessarily within the scope of Proposition \ref{Pcohomologicaltest}.

Since the image of $ \Omega \rightarrow \Omega_{\ad} $ is trivial when $ G_{\sss} $ is simply-connected, the conclusion of Theorem \ref{Tfixingrepresentative} is vacuously true. If $ G_{\sss} $ is adjoint then $ \Omega \rightarrow \Omega_{\ad} $ is surjective, i.e. $ \Omega $ is ``maximal'', but Theorem \ref{Tfixingrepresentative} applies. There are not many groups remaining: if $ G $ is as in \S\ref{Snotation} and $ G_{\sss} $ is neither simply-connected nor adjoint then $ G_{\sss} = \SL_n / \MU_m $ for some $ m | n $ or $ G_{\sss} $ is type D and $ Z ( G_{\sss} ) = \MU_2 $. 

I begin treatment of each of these cases:

\begin{prop}[Theorem \ref{Tfixingrepresentative} for non-adjoint Type A] \label{PtypeAisok}
If $ G_{\sss} $ is type A then the conclusion of Theorem \ref{Tfixingrepresentative} is true.
\end{prop}

\begin{proof}
By the proof of Theorem \ref{Tfixingrepresentative}, I may assume that $ G = G_{\sss} $. By Proposition \ref{Pcohomologicaltest}, it suffices to show that 
$ \del [ \lambda, \omega ] = 0 $ for all $ \omega \in \Omega $. Necessarily $ G_{\ad} = \PGL_n $ and there are $ m, d \in \N $ with $ n = m d $ such that $ G \defeq \SL_n / \MU_m $ and $ Z ( G ) = \MU_d $. Thus, there is a short-exact-sequence $ 1 \rightarrow \MU_d \rightarrow G \rightarrow \PGL_n \rightarrow 1 $. Let $ \specialfiber{T} $ and $ \specialfiber{T}_{\ad} $ be the $ \resfield $-split maximal diagonal tori in the obvious $ \resfield $-analogues of $ G $ and $ \PGL_n $. In the long-exact-sequence associated to $ 1 \rightarrow  \MU_d \rightarrow \specialfiber{T} \rightarrow \specialfiber{T}_{\ad} \rightarrow 1 $, there is the exact subsequence $ \specialfiber{T} ( \resfield ) \rightarrow \specialfiber{T}_{\ad} ( \resfield ) \stackrel{\del}{\longrightarrow} H^1 ( \resfield, \MU_d ) $. After using the Kummer Isomorphism, $ \del $ is induced by the determinant $ \det_{\resfield} $. Thus, similar to Example \ref{EXconnectingmapischaracter} above, it is clear that $ \det_{\resfield} ( \boldsymbol{t} ) = \alpha_1 ( \boldsymbol{t} ) \cdot \alpha_2 ( \boldsymbol{t} )^2 \cdots \alpha_{\ell} ( \boldsymbol{t} )^{\ell} $ in $ \resfield^{\times} / ( \resfield^{\times} )^d $ for any $ \boldsymbol{t} \in \specialfiber{T}_{\ad} ( \resfield ) $. Now, fix $ \omega \in \Omega $, let the notation be as in Theorem \ref{Tfixingrepresentative}, and let $ \boldsymbol{t}_{\ad} \in \specialfiber{T}_{\ad} ( \resfield ) $ be the solution to system (\ref{Efixersystem}) constructed in Theorem \ref{Tfixingrepresentative} for $ \PGL_n $. As explained by the discussion preceding Proposition \ref{Pcohomologicaltest}, $ \del [ \lambda, \omega ] = \del ( \boldsymbol{t}_{\ad} ) $ may be computed directly from the system (\ref{Efixersystem}) that $ \boldsymbol{t}_{\ad} $ solves: it is the image in $ \resfield^{\times} / ( \resfield^{\times} )^d $ of $ L \cdot L^{\prime} \cdot C $ for $ L = \lambda_1^{-1} \cdot ( \lambda_{2}^{-1} )^2 \cdots ( \lambda_{\ell}^{-1} )^{\ell} $, and $ L^{\prime} = \lambda_{\sigma( 1 )} \cdot ( \lambda_{\sigma ( 2 )} )^2 \cdots ( \lambda_{\sigma ( \ell )} )^{\ell} $, and $ C = c ( n_{\finiteweylsubscript}, \alpha_1 ) \cdot c ( n_{\finiteweylsubscript}, \alpha_2 )^2 \cdots c ( n_{\finiteweylsubscript}, \alpha_{\ell} )^{\ell} $. Thus, the claim is simply that the image of $ L \cdot L^{\prime} \cdot C $ in $ \resfield^{\times} $ is a $ d $th power. I now adopt the notation of Plate I \cite{bourbakiII}. I may assume that $ \omega \in \Omega \cong \Z / m \Z $ is the generator such that the permutation of $ \simpleaffinegradients $ by $ \sigma := \prW ( \omega ) $ is $ \alpha_i \mapsto \alpha_{i+d} $ (the subscripts are to be interpreted modulo $ n $). It is then immediate that $ L \cdot L^{\prime} $ is a $ d $th-power since, noting that $ \sigma ( 0 ) = d $ and $ \sigma ( n - d ) = 0 $, it is the product of $ ( \lambda_d^{-1} )^d $ and $ ( \lambda_0 )^{n-d} = ( ( \lambda_0 )^{m - 1} )^d $ with $ ( \lambda_{\sigma ( e )}^{-1} )^{\sigma ( e )} \cdot ( \lambda_{\sigma ( e )} )^{e} = ( \lambda_{\sigma ( e )} )^{-d} $ for all $ e $ such that $ \sigma ( e ) \neq 0 $ (i.e. all $ e \neq n - d $). I may also assume that the realization of $ \roots $ in $ G $ is a Chevalley Realization, in which case $ C = c ( n_{\finiteweylsubscript}, \alpha_{n-d} )^{n-d} = ( ( \pm 1 )^{m-1} )^d $, as desired.
\end{proof}

Among the only remaining family, type D, which contains groups that are not adjoint and not simply-connected, the following treats the easier situation:

\begin{prop}[Theorem \ref{Tfixingrepresentative} for non-adjoint Type D, case 1] \label{PtypeSOisok}
If $ G_{\sss} = \SO ( 2 \ell ) $ then the conclusion of Theorem \ref{Tfixingrepresentative} is true.
\end{prop}

\begin{proof}
Let $ \roots $ be type D. Follow Plate IV \cite{bourbakiII}, so that $ \rank ( \roots ) = \ell $ and the ``ambient'' basis is $ \epsilon_1, \ldots, \epsilon_{\ell} $ and $ \simpleroots = \{ \alpha_1, \ldots, \alpha_{\ell} \} $ where $ \alpha_i = \epsilon_i - \epsilon_{i+1} $ for $ i = 1, \ldots, \ell - 1 $ and $ \alpha_{\ell} = \epsilon_{\ell-1} + \epsilon_{\ell} $. It happens that the fundamental weight $ \omega_1 $ from Plate IV (VI) \cite{bourbakiII} is $ \omega_1 = \epsilon_1 $. Define $ X $ to be the subgroup of the weight lattice generated by $ \roots $ and $ \omega_1 = \epsilon_1 $, which is clearly $ \myspan_{\Z} ( \epsilon_1, \ldots, \epsilon_{\ell} ) $. The group $ G $ with algebraic fundamental group $ X / Q $ is $ \SO ( 2 \ell ) $ and if $ \omega \in \Omega \cong \Z / 2 \Z $ is the non-trivial element then $ \sigma \defeq \prW ( \omega ) $ has disjoint cycle decomposition $ ( - \highestroot, \alpha_1 ) ( \alpha_{\ell-1}, \alpha_{\ell} ) $ as a permutation of $ \simpleaffinegradients $. Set $ n_{\finiteweylsubscript} \defeq \canonicalsection ( \sigma ) $. I may assume that the realization of $ \roots $ in $ G $ is a Chevalley Realization, in which case system (\ref{Efixersystem}) is
\begin{align*}
\alpha_{1} ( t ) &= \lambda_{1}^{-1} \cdot \lambda_{0} \cdot c ( n_{\finiteweylsubscript}, \alpha_{1} ) \\
\alpha_j ( t ) &= 1 \textaftermath{for all $ 2 \leq j \leq \ell - 2 $} \\
\alpha_{\ell-1} ( t ) &= \lambda_{\ell-1}^{-1} \cdot \lambda_{\ell} \cdot 1 \\
\alpha_{\ell} ( t ) &= \lambda_{\ell}^{-1} \cdot \lambda_{\ell-1} \cdot 1
\end{align*}
The cocharacter lattice $ X^{\vee} $ is $ X^{\vee} = \myspan_{\Z} ( \epsilon_1^{\vee}, \ldots, \epsilon_{\ell}^{\vee} ) $ and so the values $ \epsilon_i ( t ) $ may be chosen freely for $ t \in T ( \field ) $. Thus, a solution $ t $ is created by choosing $ \epsilon_{\ell} ( t ) = \lambda_{\ell}^{-1} \cdot \lambda_{\ell-1} $, $ \epsilon_{\ell-1} ( t ) = \cdots = \epsilon_2 ( t ) = 1 $, and $ \epsilon_1 ( t ) = \lambda_{1}^{-1} \cdot \lambda_{0} \cdot c ( n_{\finiteweylsubscript}, \alpha_{1} ) $.
\end{proof}

Finally, the following treats the only remaining class of $ G $ for which $ G_{\sss} $ is not adjoint and not simply-connected:

\begin{prop}[Theorem \ref{Tfixingrepresentative} for non-adjoint Type D, case 2] \label{PtypeHalfSpinisok}
If $ G_{\sss} $ is a ``Half-Spin Group''\footnote{The main challenge here is to avoid learning what a Half-Spin Group really is...} then the conclusion of Theorem \ref{Tfixingrepresentative} is true.
\end{prop}

\begin{proof}
Let $ \roots $ be type D with rank $ \ell \in 2 \Z $, say $ \ell = 2 L $. By Proposition \ref{PtypeSOisok}, I may assume that $ \ell > 4 $. Follow Plate IV \cite{bourbakiII}, so that the ``ambient'' basis is $ \epsilon_1, \ldots, \epsilon_{\ell} $ and $ \simpleroots = \{ \alpha_1, \ldots, \alpha_{\ell-1}, \alpha_{\ell} \} $, where $ \alpha_i = \epsilon_i - \epsilon_{i+1} $ for $ 1 \leq i < \ell $ and $ \alpha_{\ell} = \epsilon_{\ell-1} + \epsilon_{\ell} $.

Set $ \omega_{\ell} = \frac{1}{2} ( \epsilon_1 + \cdots + \epsilon_{\ell} ) $, a fundamental weight, and define $ X $ be the subgroup of the weight lattice generated by $ \roots $ and $ \omega_{\ell} $. Let $ G $ be the group corresponding to $ X $, i.e. the quotient of $ \spin ( 2 \ell ) $ with whose algebraic fundamental group is $ X / Q $. By construction, $ Z ( G ) = \MU_2 $ and $ G $ is not isomorphic to $ \SO ( 2 \ell ) $, i.e. $ G $ is a ``half-spin group''.

First, determine a nice basis for $ X $. The identities $ \alpha_1 + 2 \alpha_2 + \cdots + ( \ell - 1 ) \alpha_{\ell-1} = 2 \omega_{\ell} - \ell \epsilon_{\ell} $ and $ L ( \alpha_{\ell} - \alpha_{\ell-1} ) = \ell \epsilon_{\ell} $ together imply that $ \alpha_1 \in \myspan_{\Z} ( \omega_{\ell}, \alpha_2, \ldots, \alpha_{\ell} ) $. This means that $ \omega_{\ell}, \alpha_2, \ldots, \alpha_{\ell} $ is a basis for the character lattice $ X $. Let $ \beta_1, \ldots, \beta_{\ell} $ be the basis dual to $ \omega_{\ell}, \alpha_2, \ldots, \alpha_{\ell} $, so that $ \langle \alpha_i, \beta_i \rangle = 1 = \langle \omega_{\ell}, \beta_1 \rangle $ for all $ i = 2, \ldots, \ell $ and all other pairings are $ 0 $. This is a basis for the cocharacter lattice $ X^{\vee} $. It can be computed that, in $ \epsilon $-coordinates, 
\begin{align*}
\beta_1 &= ( 2, 0, \ldots, 0 ) \\
\beta_j &= ( - j + 1, 1, \ldots, 1, 0, \ldots, 0 ) \textaftermath{for all $ 2 \leq j \leq \ell - 2 $} \\
\beta_{\ell-1} &= \left( - \frac{\ell-3}{2}, \frac{1}{2}, \ldots, \frac{1}{2}, - \frac{1}{2} \right) \\
\beta_{\ell} &= \left( - \frac{\ell-1}{2}, \frac{1}{2}, \ldots, \frac{1}{2}, \frac{1}{2} \right) \\
\end{align*}
\emph{For $ 2 \leq j \leq \ell - 2 $, the $ \epsilon_i $-coefficient of $ \beta_j $ is $ 0 $ if and only if $ i > j $.}

Let $ \omega \in \Omega \cong \Z / 2 \Z $ be the non-trivial element and set $ \sigma \defeq \prW ( \omega ) $. How $ \sigma $ permutes $ \simpleaffinegradients $ depends on the parity of $ L $, but ignore this for now. For each $ i = 0, 1, \ldots, \ell $ denote by $ \sigma ( i ) $ the $ j $ such that $ \sigma ( \alpha_i ) = \alpha_j $. Set $ n_{\finiteweylsubscript} := \canonicalsection ( \sigma ) $, the canonical representative. I may assume that the realization of $ \roots $ in $ G $ is a Chevalley Realization.

The split maximal torus $ T \subset G $ is recovered by $ X^{\vee} \otimes_{\Z} \field^{\times} \directedisom T ( \field ) $ and arbitrary $ t \in T ( \field ) $ is represented by $ t = \beta_1 ( t_1 ) \cdots \beta_{\ell} ( t_{\ell} ) $ for arbitrary $ t_i \in \field^{\times} $. By definition of the dual basis, $ \alpha_i ( t ) = t_i $ for $ i \neq 1 $ and so all equations in system (\ref{Efixersystem}) are solvable except possibly the first. Since $ \alpha_1 ( t ) = \prod_i t_i^{\langle \alpha_1, \beta_i \rangle} $ and $ \langle \alpha_1, \beta_1 \rangle = 2 $, that first equation is
\begin{equation} \label{Ehalfspinequation1}
t_1^2 \cdot \prod_{i=2}^{\ell} t_i^{\langle \alpha_1, \beta_i \rangle} = \alpha_1 ( t ) \iseq \lambda_1^{-1} \cdot \lambda_{\sigma ( 1 )} \cdot c ( n_{\finiteweylsubscript}, \alpha_1 )
\end{equation}

Since 
\begin{itemize}
\item system (\ref{Efixersystem}) prescribes the value $ t_j = \alpha_j ( t ) = \lambda_j^{-1} \cdot \lambda_{\sigma ( j )} \cdot c ( n_{\finiteweylsubscript}, \alpha_j ) $ for all $ j > 1 $,

\item $ c ( n_{\finiteweylsubscript}, \alpha_j ) = 1 $ for all $ 2 \leq j \leq \ell - 2 $ due to the fact that $ \sigma ( \alpha_j ) = \alpha_{\ell-j} \in \simpleroots $ for all $ 2 \leq j \leq \ell - 2 $ regardless of $ \omega $ by Plate IV (XII) \cite{bourbakiII}, and

\item $ \langle \alpha_1, \beta_j \rangle = - j $ for all $ 2 \leq j \leq \ell - 2 $,
\end{itemize}
I conclude further that $ t_j^{\langle \alpha_1, \beta_j \rangle} \cdot t_{\ell-j}^{\langle \alpha_1, \beta_{\ell-j} \rangle} $ is a \emph{square} for all $ 2 \leq j \leq \ell - 2 $. Thus, equation (\ref{Ehalfspinequation1}) can be replaced by
\begin{equation} \label{Ehalfspinequation2}
t_1^2 \cdot t_{\ell-1}^{\langle \alpha_1, \beta_{\ell-1} \rangle} \cdot t_{\ell}^{\langle \alpha_1, \beta_{\ell} \rangle} \iseq \lambda_1^{-1} \cdot \lambda_{\sigma ( 1 )} \cdot c ( n_{\finiteweylsubscript}, \alpha_1 )
\end{equation}

Since
\begin{itemize}
\item $ \langle \alpha_1, \beta_{\ell-1} \rangle = - ( \ell - 2 ) / 2 = - L + 1 $,

\item $ \langle \alpha_1, \beta_{\ell} \rangle = - \ell / 2 = - L $, and

\item $ c ( n_{\finiteweylsubscript}, \alpha_1 ) = 1 $ due to the fact that $ \sigma ( \alpha_1 ) \in \simpleroots $ regardless of $ \omega $,
\end{itemize}
equation (\ref{Ehalfspinequation2}) reduces to
\begin{equation} \label{Ehalfspinequation3}
t_1^2 \cdot ( \lambda_{\ell-1}^{-1} \cdot \lambda_{\sigma ( \ell-1 )} \cdot c ( n_{\finiteweylsubscript}, \alpha_{\ell-1} ) )^{- L + 1 } \cdot ( \lambda_{\ell}^{-1} \cdot \lambda_{\sigma ( \ell )} \cdot c ( n_{\finiteweylsubscript}, \alpha_{\ell} ) )^{- L} \iseq \lambda_1^{-1} \cdot \lambda_{\sigma ( 1 )}
\end{equation}

To fully understand (\ref{Ehalfspinequation3}), it is necessary finally to know how $ \sigma $ permutes $ \simpleaffinegradients $, and this depends on the parity of $ L $. If $ L \in 2 \Z $ then it is clear that the fundamental coweight $ \omega_{\ell}^{\vee} = ( \frac{1}{2}, \ldots, \frac{1}{2}, \frac{1}{2} ) $, in $ \epsilon^{\vee} $-coordinates, is a $ \Z $-combination of $ \beta_1 $ and $ \beta_{\ell} $, i.e. $ \omega_{\ell}^{\vee} \in X^{\vee} $. If $ L \notin 2 \Z $ then it is clear that $ \omega_{\ell-1}^{\vee} = ( \frac{1}{2}, \ldots, \frac{1}{2}, - \frac{1}{2} ) $ is a $ \Z $-combination of $ \beta_1 $ and $ \beta_{\ell-1} $, i.e. $ \omega_{\ell-1}^{\vee} \in X^{\vee} $. By Plate IV (XII) \cite{bourbakiII}, this means that if $ L \in 2 \Z $ then the permutation of $ \simpleaffinegradients $ by $ \sigma $ has the disjoint cycle decomposition $ ( - \highestroot, \alpha_{\ell} ) ( \alpha_1, \alpha_{\ell-1} ) \cdots ( \alpha_j, \alpha_{\ell - j} ) \cdots $, and if $ L \notin 2 \Z $ then the permutation of $ \simpleaffinegradients $ by $ \sigma $ has the disjoint cycle decomposition $ ( - \highestroot, \alpha_{\ell-1} ) ( \alpha_1, \alpha_{\ell} ) \cdots ( \alpha_j, \alpha_{\ell - j} ) \cdots $.

Thus, if $ L \in 2 \Z $ then the equation that must be solved is
\begin{align*}
t_1^2 &{\iseq} \lambda_1^{-1} \cdot \lambda_{\ell-1} \cdot ( \lambda_{\ell-1}^{-1} \cdot \lambda_{1} )^{L-1} \cdot ( \lambda_{\ell}^{-1} \cdot \lambda_{0} \cdot c ( n_{\finiteweylsubscript}, \alpha_{\ell} ) )^{L} \\
&{\iseq} ( \lambda_{\ell-1}^{-1} \cdot \lambda_{1} )^{L-2} \cdot ( \lambda_{\ell}^{-1} \cdot \lambda_{0} \cdot c ( n_{\finiteweylsubscript}, \alpha_{\ell} ) )^{L}
\end{align*}
while if $ L \notin 2 \Z $ then the equation that must be solved is
\begin{align*}
t_1^2 &{\iseq} \lambda_1^{-1} \cdot \lambda_{\ell} \cdot ( \lambda_{\ell-1}^{-1} \cdot \lambda_{0} \cdot c ( n_{\finiteweylsubscript}, \alpha_{\ell-1} ) )^{L-1} \cdot ( \lambda_{\ell}^{-1} \cdot \lambda_{1} )^{L} \\
&{\iseq} ( \lambda_{\ell-1}^{-1} \cdot \lambda_{0} \cdot c ( n_{\finiteweylsubscript}, \alpha_{\ell-1} ) )^{L-1} \cdot ( \lambda_{\ell}^{-1} \cdot \lambda_{1} \cdot c ( n_{\finiteweylsubscript}, \alpha_{\ell} ) )^{L-1}
\end{align*}
In each case, the right-hand-side is a square, which shows that a solution $ t \in T ( \field ) $ to system (\ref{Efixersystem}) exists.
\end{proof}

For completeness, I formally update the Theorem to include these cases:

\begin{theorem}[Update of Theorem \ref{Tfixingrepresentative}] \label{Tupdatedtheorem}
With no assumptions on $ G $ beyond those from \S\ref{Snotation}, there exists a section $ \goodsectionext $ of the canonical map $ N_G(T)(\field) \rightarrow W $ such that $ \goodsectionext ( \Omega ) \subset \fixer_{G_{\barycenter}} ( \lambda ) $.
\end{theorem}

\subsection{Some last refinements} \label{SSrefinements}

In this subsection, the material from \S\ref{SSgoodrepresentativesadjointcase} and \S\ref{SSgoodrepresentativesnonadjointcase} is combined to give an explicit presentation of $ \fixer_{G_{\barycenter}} ( \lambda ) $.

\begin{prop} \label{Pproductdecompofcharacterfixer}
$ \fixer_{G_{\barycenter}} ( \lambda ) = \goodsectionext ( \Omega ) \cdot Z_c \cdot \prounipotentiwahori $.
\end{prop}

Recall from \S\ref{SSadditionalnotation} that $ \integralmodel{G} $ is an $ \integers $-model of $ G $ and $ \specialfiber{G} = \integralmodel{G} \otimes_{\integers} \resfield $.

\begin{proof}
As remarked previously, it is clear that $ Z_c, \prounipotentiwahori \subset \fixer_{G_{\barycenter}} ( \lambda ) $, and $ \goodsectionext ( \Omega ) \subset \fixer_{G_{\barycenter}} ( \lambda ) $ by definition of $ \goodsectionext $. Thus, $ \fixer_{G_{\barycenter}} ( \lambda ) \supset \mathcal{S} ( \Omega ) \cdot Z_c \cdot \prounipotentiwahori $. For the reverse inclusion, let $ g \in \fixer_{G_{\barycenter}} ( \lambda ) $ be arbitrary. By Lemma \ref{Lbasicdecompofpointfixer}, there are $ \omega \in \Omega $ and $ x \in \iwahorisubgroup $ such that $ g = \goodsectionext ( \omega ) \cdot x $. By the Iwahori Factorization for $ \iwahorisubgroup $, there is $ t \in \compacttorus $ and $ u \in \prounipotentiwahori $ (since $ \prounipotentiwahori $ and $ \iwahorisubgroup $ contain the same root subgroups) such that $ x = t \cdot u $. Altogether $ g = \goodsectionext ( \omega ) \cdot t \cdot u $ and it suffices to show only that $ t \in \compactcenter \cdot \prounipotentiwahori $. Note that $ t \in \compacttorus \cap \fixer_{G_{\barycenter}} ( \lambda ) $, since it is known already that $ \goodsectionext ( \Omega ), \prounipotentiwahori \subset \fixer_{G_{\barycenter}} ( \lambda ) $. Recall the weight decomposition $ \mathbf{V}_{\barycenter} = \bigoplus_{\alpha \in \simpleaffinegradients} \mathbf{V}_{\barycenter} ( \alpha ) $. Since ``affine generic character'' means that $ \lambda $ is non-zero on each line $ \mathbf{V}_{\barycenter} ( \alpha ) $, the fact that $ t \in \compacttorus \cap \fixer_{G_{\barycenter}} ( \lambda ) $ implies that $ \alpha ( t ) \equiv 1 \text{ mod } \uniformizer $ for all $ \alpha \in \simpleroots $. Thus, if $ \boldsymbol{t} \in \specialfiber{G} ( \resfield ) $ denotes the image of $ t \in \compacttorus = \integralmodel{T} ( \integers ) $ then $ \boldsymbol{t} \in Z ( \specialfiber{G} ) ( \resfield ) $. Since $ Z ( \integralmodel{G} ) $ is a split diagonalizable group, Hensel's Lemma\footnote{The function $ \MU_n ( \integers ) \rightarrow \MU_n ( \resfield ) $ is surjective without any restriction on $ n $ (use injectivity of Frobenius on $ \resfield $ and Hensel's Lemma).} implies that $ Z ( \integralmodel{G} ) ( \integers ) \rightarrow Z ( \integralmodel{G} ) ( \resfield ) = Z ( \specialfiber{G} ) ( \resfield ) $ is surjective. Since the congruence subgroup $ \kernel ( \integralmodel{G} ( \integers ) \rightarrow \specialfiber{G} ( \resfield ) ) $ is contained in, say, $ \prounipotentiwahori $, this means that there are $ z \in Z ( \integralmodel{G} ) ( \integers ) $ and $ u \in \prounipotentiwahori $ such that $ t = z \cdot u $. Since $ \compacttorus, \prounipotentiwahori \subset G_1 $, this implies that $ z \in \compactcenter $ so the claim follows and therefore $ \fixer_{G_{\barycenter}} ( \lambda ) \subset \mathcal{S} ( \Omega ) \cdot Z_c \cdot \prounipotentiwahori $.
\end{proof}

By Proposition 2.4(1) \cite{RY}, $ \fixer_{G_{\barycenter}} ( \lambda ) $ is an inducing subgroup for the simple supercuspidals constructed from $ \lambda $, so it is important to know that it is open and compact-mod-center. Any subgroup $ \mathcal{J} \subset G ( \field ) $ that is both open and compact-mod-center necessarily contains a compact open subgroup $ J = \mathcal{J} \cap G^1 $ which is ``largest'' in the sense that it contains all other compact subgroups of $ \mathcal{J} $. For $ \fixer_{G_{\barycenter}} ( \lambda ) $, this largest compact subgroup can be presented succinctly in terms of the factorization from Proposition \ref{Pproductdecompofcharacterfixer} above:

\begin{corollary}\label{Cfactorizationoflargestcompact}
$ \fixer_{G_{\barycenter}} ( \lambda ) $ is open and compact-mod-center and if $ J \subset \fixer_{G_{\barycenter}} ( \lambda ) $ denotes the largest compact open subgroup then $ J = \goodsectionext ( \Omega_{\tor} ) \cdot Z_c \cdot \prounipotentiwahori $, where $ \Omega_{\tor} \subset \Omega $ is the torsion subgroup.
\end{corollary}


\begin{proof}
Let $ H $ be the subgroup $ Z ( G ) ( \field ) \cdot \prounipotentiwahori \subset \fixer_{G_{\barycenter}} ( \lambda ) $, which is obviously open and compact-mod-center. Set $ n := \# \Omega_{\ad} $. Since $ \goodsectionext ( \omega )^n \in T ( \field ) $ for all $ \omega \in \Omega $ and since $ Z ( G ) ( \field ) \cdot G_1 $ is a finite-index normal subgroup of $ G ( \field ) $, there is $ N > 0 $ such that $ \goodsectionext ( \omega )^N \in T ( \field ) \cap ( Z ( G ) ( \field ) \cdot G_1 ) $ for all $ \omega \in \Omega $. It follows immediately that $ \goodsectionext ( \omega )^N \in Z ( G ) ( \field ) \cdot \compacttorus $ for all $ \omega \in \Omega $. Since $ \Omega $ is finitely-generated and abelian, since $ \goodsectionext $ is homomorphic modulo $ \compacttorus $, since $ \compacttorus / \onemodptorus $ is finite, and since $ \onemodptorus \subset H $, the previous fact implies that $ H $ is a finite-index subgroup and therefore that $ \fixer_{G_{\barycenter}} ( \lambda ) $ is also open and compact-mod-center.

Let $ g \in J $ be arbitrary. By Proposition \ref{Pproductdecompofcharacterfixer}, there are $ \omega \in \Omega $, $ z \in Z_c $, $ u \in \prounipotentiwahori $ such that $ g = \goodsectionext ( \omega ) \cdot z \cdot u $, and the claim is that $ \omega \in \Omega_{\tor} $. Since $ Z_c \cdot \prounipotentiwahori $ is certainly a compact open subgroup, $ Z_c \cdot \prounipotentiwahori \subset J $. This implies that $ \goodsectionext ( \omega )^n \in J $ for all $ n \in \N $. The cosets of $ J $ modulo $ Z_c \cdot \prounipotentiwahori $ constitute an open cover, so compactness of $ J $ (and the Pigeonhole Principle) implies that $ \goodsectionext ( \omega )^N \in Z_c \cdot \prounipotentiwahori $ for some $ N > 0 $. Since both $ Z_c, \prounipotentiwahori \subset G_1 $, it is true that $ \goodsectionext ( \omega )^N \in G_1 $. Also, $ \goodsectionext ( \omega ) \in G_{\barycenter} $. Together, $ \goodsectionext ( \omega )^N \in N_G(T)(\field) \cap G_1 \cap G_{\barycenter} $. Since $ N_G(T)(\field) \cap G_1 \cap G_{\barycenter} = \compacttorus $, this means that $ \goodsectionext ( \omega^N ) $ represents $ 1_{\Omega} $ and so $ \omega^N = 1_{\Omega} $. This establishes $ J \subset \goodsectionext ( \Omega_{\tor} ) \cdot Z_c \cdot \prounipotentiwahori $. For the reverse inclusion\footnote{The inclusion $ J \supset \goodsectionext ( \Omega_{\tor} ) \cdot Z_c \cdot \prounipotentiwahori $ does not seem to be automatic from the definition of $ J $ because, while the product is obviously compact and open, it is a priori not clear that it is a \emph{subgroup}.}, it suffices to show that $ \goodsectionext ( \omega ) \in J $ for all $ \omega \in \Omega_{\tor} $. If $ n = \vert \omega \vert $ then, since $ \goodsectionext $ is a section of the canonical map, $ \goodsectionext ( \omega )^n \in \compacttorus $. Since $ \compacttorus \subset G_1 \subset G^1 $ and $ J = \fixer_{G_{\barycenter}} ( \lambda ) \cap G^1 $, it follows that $ \goodsectionext ( \omega )^n \in J $. It is obvious that the set $ \goodsectionext ( \Omega ) $ normalizes $ J $, since $ \goodsectionext ( \Omega ) \subset \fixer_{G_{\barycenter}} ( \lambda ) $ and $ J $ is the unique maximal compact open subgroup of $ \fixer_{G_{\barycenter}} ( \lambda ) $. Together, this implies that the generated subgroup $ \langle \goodsectionext ( \omega ), J \rangle $ is still compact and therefore $ \langle \goodsectionext ( \omega ), J \rangle \subset J $. Thus, $ \goodsectionext ( \omega ) \in J $, as desired.
\end{proof}


\end{document}